\documentclass[12pt, reqno]{amsart}

\usepackage{amssymb, mathrsfs}
\usepackage{latexsym}
\usepackage{amsmath, commath}
\usepackage{amsthm}
\usepackage{amsfonts}
\usepackage{dsfont}
\usepackage{mathtools}
\usepackage{epsfig}
\usepackage{verbatim}
\usepackage{graphicx}

\usepackage{tikz}
\usetikzlibrary{cd}
\usetikzlibrary{patterns}
\usetikzlibrary{calc}
\usepackage{pgfplots}
\pgfplotsset{compat=1.11}

\usepackage[left=1.5in,right=1.5in]{geometry}
\usepackage{enumerate}
\usepackage[colorlinks=true,linkcolor=blue,citecolor=blue]{hyperref}

\usepackage{soul}

\usepackage{todonotes}

\usepackage[titletoc, title]{appendix}

\usepackage{bm}
\usepackage{bbm}
\usepackage{nicefrac}

\usepackage{slashed}

\newtheorem*{lemma*}{Lemma}

\newtheorem{theorem}{Theorem}[section]
\newtheorem{proposition}[theorem]{Proposition}
\newtheorem{lemma}[theorem]{Lemma}

\newtheorem{corollary}[theorem]{Corollary}

\newtheorem{thmx}{Theorem}

\newtheorem{conjecture}{Conjecture}
\newtheorem*{open question}{Open Question}

\newtheorem{question}[theorem]{Question}
\newtheorem*{question*}{Question}

\theoremstyle{definition}
\newtheorem{claim}[theorem]{Claim}
\newtheorem*{claim*}{Claim}
\newtheorem{definition}[theorem]{Definition}
\newtheorem{example}[theorem]{Example}

\theoremstyle{remark}
\newtheorem{remark}[theorem]{Remark}

\newcommand{\interior}[1]{{\kern0pt#1}^{\mathrm{o}}}

\newcommand{\ind}{\textup{Ind}}
\newcommand{\supp}{\textup{supp}}

\newcommand{\width}{\textup{width}}
\newcommand{\dist}{\textup{dist}}

\newcommand{\Sc}{\mathrm{Sc}}

\newcommand{\hotimes}{\mathbin{\widehat{\otimes}}}

\newcommand{\sone}{\mathbb S^1}

\DeclareMathOperator{\Ran}{\mathcal R}

\DeclareMathOperator{\dom}{Dom}
\DeclareMathOperator{\End}{\mathrm{End}}
\newcommand{\cpto}{\mathcal K}

\newcommand{\spinb}{\mathcal S} 
\newcommand{\cl}{ \textup{C}\ell} 
\newcommand{\ccl}{\hspace{1pt}\mathbb C\ell}  
\newcommand{\spin}{\mathrm{Spin}}

\newcommand{\op}{\mathrm{op}}

\newcommand{\dirac}{|\bm{D}|}

\newcommand{\sphere}{\mathbb S}
\newcommand{\equator}{\mathbb E}

\newcommand{\torus}{\mathbb T}
\newcommand{\double}[1]{\mathfrak{#1}}

\newcommand{\id}{\mathbbm{1}}
\newcommand{\bott}{\mathfrak{p}}
\newcommand{\botts}{\bm{p_{n}}}
\newcommand{\bottone}{\bm{v_n}}
\newcommand{\bottu}{\mathfrak{v}}
\newcommand{\idsp}{\bm{1}}
\newcommand{\ch}{\mathrm{ch}}
\newcommand{\SO}{\mathrm{SO}}
\newcommand{\prop}{\textup{prop}}

\newcommand{\fried}{|\widetilde D|}
\newcommand{\R}{\mathbb R}
\newcommand{\sgn}{\mathrm{sgn}}

\makeatletter
\let\save@mathaccent\mathaccent
\newcommand*\if@single[3]{%
	\setbox0\hbox{${\mathaccent"0362{#1}}^H$}%
	\setbox2\hbox{${\mathaccent"0362{\kern0pt#1}}^H$}%
	\ifdim\ht0=\ht2 #3\else #2\fi
}
\newcommand*\rel@kern[1]{\kern#1\dimexpr\macc@kerna}
\newcommand*\overbar[1]{\@ifnextchar^{{\wide@bar{#1}{0}}}{\wide@bar{#1}{1}}}
\newcommand*\wide@bar[2]{\if@single{#1}{\wide@bar@{#1}{#2}{1}}{\wide@bar@{#1}{#2}{2}}}
\newcommand*\wide@bar@[3]{%
	\begingroup
	\def\mathaccent##1##2{%
		\let\mathaccent\save@mathaccent
		\if#32 \let\macc@nucleus\first@char \fi
		\setbox\z@\hbox{$\macc@style{\macc@nucleus}_{}$}%
		\setbox\tw@\hbox{$\macc@style{\macc@nucleus}{}_{}$}%
		\dimen@\wd\tw@
		\advance\dimen@-\wd\z@
		\divide\dimen@ 3
		\@tempdima\wd\tw@
		\advance\@tempdima-\scriptspace
		\divide\@tempdima 10
		\advance\dimen@-\@tempdima
		\ifdim\dimen@>\z@ \dimen@0pt\fi
		\rel@kern{0.6}\kern-\dimen@
		\if#31
		\overline{\rel@kern{-0.6}\kern\dimen@\macc@nucleus\rel@kern{0.4}\kern\dimen@}%
		\advance\dimen@0.4\dimexpr\macc@kerna
		\let\final@kern#2%
		\ifdim\dimen@<\z@ \let\final@kern1\fi
		\if\final@kern1 \kern-\dimen@\fi
		\else
		\overline{\rel@kern{-0.6}\kern\dimen@#1}%
		\fi
	}%
	\macc@depth\@ne
	\let\math@bgroup\@empty \let\math@egroup\macc@set@skewchar
	\mathsurround\z@ \frozen@everymath{\mathgroup\macc@group\relax}%
	\macc@set@skewchar\relax
	\let\mathaccentV\macc@nested@a
	\if#31
	\macc@nested@a\relax111{#1}%
	\else
	\def\gobble@till@marker##1\endmarker{}%
	\futurelet\first@char\gobble@till@marker#1\endmarker
	\ifcat\noexpand\first@char A\else
	\def\first@char{}%
	\fi
	\macc@nested@a\relax111{\first@char}%
	\fi
	\endgroup
}
\makeatother

\makeatletter
\newsavebox{\@brx}
\newcommand{\llangle}[1][]{\savebox{\@brx}{\(\m@th{#1\langle}\)}%
	\mathopen{\copy\@brx\kern-0.5\wd\@brx\usebox{\@brx}}}
\newcommand{\rrangle}[1][]{\savebox{\@brx}{\(\m@th{#1\rangle}\)}%
	\mathclose{\copy\@brx\kern-0.5\wd\@brx\usebox{\@brx}}}
\makeatother

\numberwithin{equation}{section}

\begin{document}
\title[Quantitative relative index and Gromov's conjectures on psc]{A quantitative relative index theorem and Gromov's conjectures on positive scalar curvature}

\author[Zhizhang Xie]{Zhizhang Xie  \\  \mbox{} \\ (with appendices by Jinmin Wang and Zhizhang Xie)} 
\address{Department of Mathematics, Texas A\&M University}
\email{xie@math.tamu.edu}
\thanks{Partially supported by NSF 1800737 and NSF 1952693. }
\date{}

	\maketitle

	\begin{abstract}
		In this paper, we prove a quantitative relative index theorem. It provides a conceptual framework for studying some conjectures and open questions of Gromov on positive scalar curvature. More precisely,  we prove  a $\lambda$-Lipschitz rigidity theorem for (possibly incomplete) Riemannian metrics on  spheres with certain types of subsets removed. This  $\lambda$-Lipschitz rigidity theorem is asymptotically  optimal.  As a consequence, we obtain an asymptotically optimal $\lambda$-Lipschitz rigidity theorem for positive scalar curvature metrics on hemispheres.   These give positive answers to  the corresponding open questions raised by Gromov. As another application,   we prove Gromov's $\square^{n-m}$ inequality  on the bound of distances between opposite faces of spin manifolds with cube-like boundaries with a suboptimal constant. As immediate consequences, this implies Gromov's cube inequality on the bound of widths of Riemannian cubes  and Gromov's conjecture on the bound of widths of Riemannian bands with suboptimal constants.  Further geometric applications will be discussed in a forthcoming paper.
	\end{abstract}

	\section{Introduction}

In the past several years, Gromov has formulated an extensive list of conjectures and
open questions on  scalar curvature \cite{MR3816521, Gromov:2019aa}. This has 
given rise to new perspectives on scalar curvature and inspired a wave of recent activity 
in this area \cite{MR4181824, Cecchini:2021vs,  Chodosh:2020tk, MR3816521, Gromov:2019aa,Gromov:2020aa, Guo:2020ur, MR4050100, Lott:2020vq,  Wang:2021uj, Zeidler:2019aa,MR4181525}.	In this paper, we develop a quantitative relative index theorem that serves as a conceptual framework for solving  some of these conjectures and open questions. 

 For example,   we answer the following conjecture of Gromov in the spin case for all dimensions with a suboptimal constant.

	\begin{conjecture}[{Gromov's $\square^{n-m}$ inequality, \cite[section 5.3]{Gromov:2019aa}}]\label{conj:gcube}
		Let $(X, g)$ be an $n$-dimensional  compact connected orientable manifold with boundary and $\underline{X}_\bullet$ a closed orientable manifold of dimension $n-m$. Suppose
		\[  f\colon X\to [-1, 1]^m\times \underline{X}_\bullet \]	
		is a continuous map, which sends the boundary of $X$ to the boundary of $[-1, 1]^m\times \underline{X}_\bullet$ and which has non-zero degree.
		Let $\partial_{j\pm}, j = 1, \dots, m$,  be the pullbacks of the pairs of the opposite faces of the cube $[-1, 1]^m$ under the composition of $f$ with the projection $[-1, 1]^m\times \underline{X}_\bullet \to [-1, 1]^m$.
		Assume that for any $m$ hypersurfaces $Y_j\subset X$ that separate $\partial_{j-}$ from $\partial_{j+}$ with $1\leq j \leq m$,   their transversal intersection $Y_\pitchfork\subset X$ does not admit a metric with positive scalar curvature; furthermore, the products $Y_{\pitchfork}\times T^k$ of $Y_{\pitchfork}$ and $k$-dimensional tori do not admit metrics with positive scalar curvature either.  If $\Sc(g) \geq n(n-1)$, then the distances $\ell_j = \dist(\partial_{j_-}, \partial_{j_+})$ satisfy the following inequality:
		\[ \sum_{j=1}^m \frac{1}{\ell_j^2} \geq \frac{n^2}{4\pi^2}. \]
		Consequently, we have 
		\[  \min_{1\leq j \leq m} \dist(\partial_{j_-}, \partial_{j_+}) \leq \sqrt{m} \frac{2\pi}{n}.\]
		
	\end{conjecture}
	Here if $(X, g)$ is a manifold with Riemannian metric $g$, then $\Sc(g)$ stands for the scalar curvature of $g$. Sometimes, we also  write $\Sc(X)$ for the scalar curvature of $g$ if it is clear from the context which metric we are referring to. In \cite{Gromov:2019aa}, Gromov gave a proof of Conjecture \ref{conj:gcube} using Schoen-Yau's minimal surface method \cite{RSSY79b, RSSY79, Schoen:2017aa}. In dimension $\geq 9$,  Gromov's proof relies unpublished 
	results  of Lohkamp \cite{Lohkamp:2018wp} or a generalization of   Schoen-Yau's result  \cite{Schoen:2017aa}. He stated that the above inequalities should be regarded as conjectural  in dimension $\geq 9$, cf.  \cite[Section 5.2, page 250]{Gromov:2019aa}.

	The conditions in Conjecture $\ref{conj:gcube}$ may appear technical at the first glance. 
	The following special case probably makes it clearer what kind of geometric problems we are dealing with here. 
	
	\begin{conjecture}[{Gromov's $\square^n$ inequality, \cite[section 3.8]{Gromov:2019aa}}]\label{conj:cube}
		Let $g$ be a Riemannian metric on the cube $I^n = [0, 1]^n$. If $\Sc(g) \geq n(n-1)$, then 
		\[ \sum_{j=1}^n \frac{1}{\ell^2_j} \geq \frac{n^2}{4\pi^2}, \]
		where $\ell_j = \dist(\partial_{j_-}, \partial_{j_+})$ is the $g$-distance between the pair of opposite faces $\partial_{j_-}$ and $\partial_{j_+}$ of the cube. 
		Consequently, we have 
		\[ \min_{1\leq j \leq n} \dist(\partial_{j_-}, \partial_{j_+}) \leq \frac{2\pi}{\sqrt{n}}. \]
	\end{conjecture}

 One of the  key ingredients for the proof of Conjecture \ref{conj:gcube} is the following quantitative relative index theorem.

	\begin{thmx}[cf. Theorem \ref{thm:maxrelative}] \label{thm:relative-intro}
		Let $Z_1$  and $Z_2$ be two closed $n$-dimensional Riemannian manifold  and $\spinb_j$ a Euclidean $\cl_n$-bundle\footnote{Here $\cl_n$ is the real Clifford algebra of $\mathbb R^n$. See   \cite[II.\S 7 and III. \S 10]{BLMM89} for more details on $\cl_n$-vector bundles and the Clifford index of $\cl_n$-linear Dirac operators.} over $Z_j$ for $j=1, 2$.  Suppose $D_j$ is a  $\cl_n$-linear Dirac-type operator acting on  $\spinb_j$ over $Z_j$. Let $\widetilde Z_j$ be a Galois $\Gamma$-covering space of $Z_j$ and $\widetilde D_j$ the lift of $D_j$. Let $X_j$ be a subset of $Z_j$ and $\widetilde X_j$ the preimage of $X_j$ under the covering map $\widetilde Z_j \to Z_j$. Denote by  $N_r({Z_j\backslash X_j})$ the open $r$-neighborhood of ${Z_j\backslash X_j}$.  Suppose there is $r>0$ such that all geometric data on $N_r({Z_1\backslash X_1})$ and $N_r({Z_2\backslash X_2})$ coincide, i.e. there is an  orientation preserving Riemannian isometry $\Phi\colon N_r({Z_1\backslash X_1}) \to N_r({Z_2\backslash X_2})$ such that $\Phi$ lifts to an isometric $\cl_n$-bundle isomorphism 
		$\Phi\colon \spinb_1|_{N_r({Z_1\backslash X_1})} \to \spinb_2|_{N_r({Z_2\backslash X_2})}$. Assume that 
		\begin{enumerate}[$(1)$]
			\item  there exists $\sigma>0$ such that 
			\[ \mathcal R_j(x) \geq \frac{(n-1)\sigma^2}{n} \] for all $x\in X_j$, where $\mathcal R_j$ is the curvature term appearing in $D_j^2 = \nabla^\ast \nabla + \mathcal R_j$,
			\item and $D_1 = \Phi^{-1} D_2 \Phi$ on  $N_r({Z_1\backslash X_1})$. 
		\end{enumerate}
		Then there exists a universal constant $C>0$ such that if  $\sigma\cdot r >C$, then  we have
		\[ \ind_{\Gamma}(\widetilde D_1) - \ind_{\Gamma}(\widetilde D_2) = 0  \] 
		in $KO_{n}(C_{\max}^\ast(\Gamma; \mathbb R))$, where  $\ind_{\Gamma}(\widetilde D_j)$ denotes the maximal higher index of $\widetilde D_j$ and   $C^\ast_{\max}(\Gamma; \mathbb R)$ is the  maximal group $C^\ast$-algebra of $\Gamma$ with real coefficients.
	\end{thmx}
	
The numerical estimates  in Appendix \ref{app:estimate} show that the universal constant $C$ is $\leq 40.65.$
	
	As an application of the above quantitative relative index theorem, we solve Gromov's $\square^{n-m}$ inequality (Conjecture $\ref{conj:gcube}$) for all dimensions with  a  suboptimal constant. More precisely, we have the following theorem.

	\begin{thmx}[cf. Theorem \ref{thm:gcube2}]\label{thm:gcube}
			Let $X$ be an $n$-dimensional  compact connected spin manifold with boundary.
		Suppose 
		$f\colon X\to [-1, 1]^m$
		is  a smooth map that  sends the boundary of $X$ to the boundary of $[-1, 1]^m$. 
		 Let $\partial_{j\pm}, j = 1, \dots, m$,  be the pullbacks of the pairs of the opposite faces of the cube $[-1, 1]^m$. 
		Suppose  $Y_{\pitchfork}$ is an $(n-m)$-dimensional closed submanifold \textup{(}without boundary\textup{)} in $X$ that satisfies the following conditions:
		\begin{enumerate}[$(1)$]
			\item $\pi_1(Y_\pitchfork) \to \pi_1(X)$ is injective;
			\item $Y_{\pitchfork}$  is the transversal intersection\footnote{In particular, this implies that the normal bundle of $Y_{\pitchfork}$ is trivial.}  of $m$ orientable hypersurfaces $\{Y_j\}_{1\leq j \leq m} $ of $X$, each of which separates $\partial_{j-}$ from $\partial_{j+}$;
			\item the higher index $\ind_{\Gamma}(D_{Y_\pitchfork})$ does not vanish in $KO_{n-m}(C^\ast_{\max}(\Gamma; \mathbb R))$, where ${\Gamma = \pi_1(Y_\pitchfork)}$.  
		\end{enumerate}  
		If $\Sc(X) \geq  n(n-1)$, then the distances $\ell_j = \dist(\partial_{j-}, \partial_{j+})$ satisfy the following inequality:
	\[ \sum_{j=1}^m \frac{1}{\ell_j^2} \geq \frac{n^2}{(\frac{8}{\sqrt{3}}C + 4\pi)^2}. \]
	where $C$ is the universal constant from Theorem $\ref{thm:relative-intro}$. 	Consequently, we have 
	\[  \min_{1\leq j\leq m} \dist(\partial_{j-}, \partial_{j+}) \leq \sqrt{m} \frac{\frac{8}{\sqrt{3}}C + 4\pi}{n}.\]
	
	\end{thmx}
   Subsequently, with Wang and Yu \cite{Wang:2021um}, the author proved Theorem \ref{thm:gcube} with the optimal constant via a different method, hence completely solves Conjecture \ref{conj:gcube} (in spin case)  and Conjecture \ref{conj:cube}  for all dimensions. We point out that  Cecchini \cite{MR4181824} and Zeidler \cite{Zeidler:2019aa,MR4181525} proved a special case  of Theorem \ref{thm:gcube}  when $m=1$ with the optimal constant.

 For spin manifolds,  the assumptions on $Y_\pitchfork$ in Theorem \ref{thm:gcube} above are  (stably) equivalent to the assumptions in Conjecture $\ref{conj:gcube}$, provided that the (stable) Gromov-Lawson-Rosenberg conjecture holds for $\Gamma = \pi_1(Y_\pitchfork)$. See the survey paper of Rosenberg and Stolz \cite{MR1818778} for more details.  The stable Gromov-Lawson-Rosenberg conjecture for $\Gamma$ follows from the strong Novikov conjecture for $\Gamma$, where the latter has been verified for a large class of groups including all word hyperbolic groups \cite{CM90}, all groups acting properly and isometrically on simply connected and non-positively curved manifolds \cite{GK88}, all subgroups of linear groups \cite{MR2217050},  and all groups that are coarsely embeddable into Hilbert space \cite{GY00}.

	As a special case of Theorem \ref{thm:gcube}, we have the following  theorem, which proves Gromov's $\square^n$-inequality  (Conjecture $\ref{conj:cube}$) with a suboptimal constant.
	
	\begin{thmx}\label{thm:cube}
		Let $g$ be a Riemannian metric on the cube $I^n = [0, 1]^n$. If $\Sc(g) \geq n(n-1)$, then  
		\[ \sum_{j=1}^n \frac{1}{\ell^2_j} \geq \frac{n^2}{
			(\frac{8}{\sqrt{3}}C + 4\pi)^2 }, \]
		where $\ell_j = \dist(\partial_{j_-}, \partial_{j_+})$ is the $g$-distance between the pair of opposite faces $\partial_{j_-}$ and $\partial_{j_+}$ of the cube, and $C$ is the universal constant from Theorem $\ref{thm:relative-intro}$.
		Consequently, we have 
		\[ \min_{1\leq j \leq n} \dist(\partial_{j-}, \partial_{j+}) \leq  \frac{\frac{8}{\sqrt{3}}C + 4\pi}{\sqrt{n}}. \]
	\end{thmx}
	\begin{proof}
		Note that the higher index of the Dirac operator on a single point is a generator of $KO_0(\{e\}) = \mathbb Z$, hence does not vanish. 	If $X$ is the cube $I^n = [0, 1]^n$ with the given Riemannian metric $g$, then the assumptions of Theorem \ref{thm:gcube} are satisfied. Hence the theorem follows from Theorem \ref{thm:gcube}. 
	\end{proof}

As pointed out by Gromov in \cite[Section 3.8]{Gromov:2019aa}, Theorem \ref{thm:cube} has the following immediate corollary. Recall that a map $\varphi\colon (X, g) \to (Y, h)$ between two metric spaces is said to be $\lambda$-Lipschitz if 
\[ \dist_{h}(\varphi(x_1), \varphi(x_2)) \leq \lambda \cdot \dist_g(x_1, x_2) \]
for all $x_1, x_2 \in X$. 
\begin{corollary}\label{cor:hemisphere}
	Let $(X, g_0)$ be the standard unit hemisphere $\sphere^n_+$.  If $X$ admits a Riemannian metric $g$ such that  
	\begin{enumerate}[$(1)$]
		\item there is a $\lambda_n$-Lipschitz homeomorphism $\varphi\colon (X, g) \to (X, g_0)$,  	 
		\item and $\Sc(g) \geq n(n-1) = \Sc(g_0)$, 
	\end{enumerate}
	then  
	\[ \lambda_n \geq \frac{2\arcsin (\frac{1}{\sqrt{n}}) }{\frac{1}{\sqrt{n}}  (\frac{4}{\sqrt{3}}C + 4\pi)} > \frac{2}{\frac{4}{\sqrt{3}}C + 4\pi}, \]
	where  $C$ is the same universal constant from Theorem \ref{thm:relative-intro}.
 \end{corollary}

As another application of our quantitative relative index theorem, we prove the following $\lambda$-Lipschitz rigidity theorem for positive scalar curvature metrics on spheres with certain subsets removed.  This gives a positive answer to a corresponding open question of Gromov, cf.  \cite[page 687, specific problem]{MR3816521}.

  \begin{thmx}[cf. Theorem \ref{thm:sphere-ad2}]\label{thm:sphere-ad}
  	Let $\Sigma$  be a  subset of the standard unit sphere $\sphere^n$ contained in a geodesic ball of radius $r < \frac{\pi}{2}$.  Let $(X, g_0)$ be the standard unit sphere $\sphere^n$ minus $\Sigma$. If a \textup{(}possibly incomplete\textup{)} Riemannian metric $g$ on $X$ satisfies that  
  	\begin{enumerate}[$(1)$]
  		\item there is a $\lambda_n$-Lipschitz  homeomorphism $\varphi\colon (X, g) \to (X, g_0)$,  	 
  		\item and $\Sc(g) \geq n(n-1) = \Sc(g_0)$, 
  	\end{enumerate}
  	then 
  	\[ \lambda_n \geq \sqrt{1 - \frac{C_r}{n^2}}, \]
  	where\footnote{If $n=\dim \sphere^n$ is odd, our proof of Theorem \ref{thm:sphere-ad} in fact shows that we can improve $C_r$ to be  
  		\[ \frac{4C^2}{(\frac{\pi}{2} - r)^2} \textup{ instead of } \frac{8C^2}{(\frac{\pi}{2} - r)^2}. \]  } 
  	\[ C_r = \frac{8C^2}{(\frac{\pi}{2} - r)^2}  \]
  	and $C$ is a universal constant from Theorem $\ref{thm:relative-intro}$. Consequently, the lower bound for $\lambda_n$ approaches $1$, as $n\to \infty$. 
  \end{thmx}

 This  $\lambda$-Lipschitz rigidity theorem is asymptotically  optimal in the sense that the lower bound for $\lambda_n$ becomes sharp,  as $n = \dim\sphere^n\to \infty$. In the case where $n=\dim \sphere^n$ is odd, an analogue of Theorem \ref{thm:sphere-ad} also holds for subsets that are contained in a pair of antipodal geodesic balls of radius $r < \frac{\pi}{6}$. We refer the reader to Theorem \ref{thm:sphere-odd2} for the precise details. We point out that  when $\Sigma = \varnothing$, that is, when $(X, g_0)$ is the standard unit sphere $\sphere^n$ itself,  it is a theorem of Llarul that $\lambda_n\geq 1$ for all $n\geq 2$ \cite[theorem A]{MR1600027}.  Furthermore,  when $\Sigma$ is either a single point or a pair of antipodal points, Gromov  showed that $\lambda_n\geq 1$ when $3\leq n \leq 8$ \cite[section 3.9]{Gromov:2019aa}. 
 
 As a consequence of Theorem \ref{thm:sphere-ad}, we have the following \mbox{$\lambda$-Lipschitz} rigidity result for hemispheres. This answers (asymptotically)  an open question of Gromov on the sharpness of the constant $\lambda_n$ for the $\lambda$-Lipschitz rigidity of positive scalar curvature metrics on hemispheres \cite[section 3.8]{Gromov:2019aa}. 
  
  \begin{thmx}[cf. Theorem \ref{thm:hemisphere2}]\label{thm:hemisphere}
  	Let $(X, g_0)$ be the standard unit hemisphere $\sphere^n_+$. If a  Riemannian metric $g$ on $X$ satisfies that  
  	\begin{enumerate}[$(1)$]
  		\item there is a $\lambda_n$-Lipschitz  homeomorphism $\varphi\colon (X, g) \to (X, g_0)$,  	 
  		\item and $\Sc(g) \geq n(n-1) = \Sc(g_0)$, 
  	\end{enumerate}
  	then 
  	\[  \lambda_n \geq  (1-\sin \frac{\pi}{\sqrt{n}}) \sqrt{1 - \frac{8C^2}{\pi^2 n}} \]
  	where  $C$ is a universal constant from Theorem $\ref{thm:relative-intro}$. Consequently, the lower bound for $\lambda_n$ approaches $1$, as $n\to \infty$. 
  \end{thmx}
  The above theorem is asymptotically optimal in the sense that   the lower bound for $\lambda_n$ becomes sharp,  as $n = \dim\sphere^n\to \infty$. In particular, it  significantly improves the lower bound for  $\lambda_n$ in Corollary \ref{cor:hemisphere} when $n = \dim\sphere^n$ is  large.

  A key geometric concept behind the proof of Theorem \ref{thm:sphere-ad} is the following notion of wrapping property for subsets of $\sphere^n$. 
  
  \begin{definition}\label{def:nonsep}
 	Let $Z$ be a path-connected metric space. A subset $\Sigma$ of $Z$ is called strongly non-separating if $Z\backslash N_\varepsilon(\Sigma)$ is path-connected for all  sufficiently small $\varepsilon >0$, where $N_\varepsilon(\Sigma)$  is the open $\varepsilon$-neighborhood of $\Sigma$. 
 \end{definition}

  \begin{definition}[Subsets with the wrapping property]\label{def:wrap-intro}
  	A subset $\Sigma$ of the standard unit sphere $\sphere^n$ is said to have \emph{the wrapping property} if  $\Sigma$ is strongly non-separating and furthermore there exists a  smooth distance-contracting\footnote{Recall that a smooth map $\psi \colon X \to Y$ between Riemannian manifolds is said to be distance-contracting if it is $1$-Lipschitz, that is,  $
  		\|\psi_\ast(v)\| \leq \|v\|$ 
  		for all tangent vectors $v\in TX$.} map $\Phi\colon \sphere^n \to \sphere^n$ such that the following are satisfied:
  	\begin{enumerate}[(1)]	  
  		\item[(1a)] if $n$ is even,  $\Phi$ equals the identity map on  $N_{\varepsilon}(\Sigma)$; 
  		\item[(1b)] if $n$ is odd, $\Phi$ equals either the identity map or the antipodal map on each of the connected components of  $N_{\varepsilon}(\Sigma)$;   
  		\item[(2)] and\footnote{For example, if $\Phi$ is not surjective, then clearly $\deg(\Phi) = 0\neq 1$.} $\deg(\Phi) \neq 1$.
  	\end{enumerate}
  \end{definition} 
  Note that the conditions for satisfying the wrapping property in the odd dimensional case are slightly weaker than those in the even dimensional case.  Roughly speaking,  a subset $\Sigma\subset \sphere^n$ has the wrapping property if its geometric size is ``relatively small''. For example, if $\Sigma$ is a strongly non-separating subset of the standard unit sphere $\sphere^n$  that is  contained in a geodesic ball of radius $ < \frac{\pi}{2}$,  then $\Sigma$  has the wrapping property (cf. Lemma \ref{lm:wrap}). Moreover, a strongly non-separating subset of an odd dimensional sphere that is contained in a pair of antipodal geodesic balls of radius $<\frac{\pi}{6}$ also satisfies the wrapping property (cf. Lemma \ref{lm:wrap-antipodal}). 
  
  Motivated by the theorems of Llarul and Gromov and the results in the current paper, we conclude this introduction  by the following open question. 
  \begin{open question}[Rigidity for positive scalar curvature metrics on $\sphere^n\backslash \Sigma$]
  	Let $\Sigma$  be a subset with the wrapping property in the standard unit sphere $\sphere^n$. Let $(X, g_0)$ be the standard unit sphere $\sphere^n$ minus $\Sigma$. If a \textup{(}possibly incomplete\textup{)} Riemannian metric $g$ on $X$ satisfies 
  	\begin{enumerate}[$(1)$]
  		\item $g\geq g_0$, 
  		\item and $\Sc(g) \geq n(n-1) = \Sc(g_0)$, 
  	\end{enumerate}
  	then does it imply that $g = g_0$? 
  \end{open question}

	The paper is organized as follows. In Section \ref{sec:pre}, we review  the construction of some standard geometric $C^\ast$-algebras and the construction of higher index.  In Section \ref{sec:relative}, we prove our quantitative relative index  theorem (Theorem \ref{thm:relative} and Theorem \ref{thm:maxrelative}). Finally, we apply the quantitative relative index theorem to prove Theorems \ref{thm:gcube}--\ref{thm:hemisphere} in Section \ref{sec:gcube} and Section \ref{sec:rigidity}. 
	
	\subsection*{Acknowledgements}
	The author wants to thank Dean Baskin, Alex Engel, Nigel Higson, Yanli Song, Xiang Tang, Jinmin Wang, Guoliang Yu and Rudolf Zeidler for many stimulating discussions over the years.  In particular, the author would like to thank Alex Engel,  Jinmin Wang, Guoliang Yu and Rudolf Zeidler for helpful comments regarding some technical issues in an earlier version of the paper.

	\section{Preliminaries}\label{sec:pre}
	
	In this section, we review  the construction of some standard geometric $C^*$-algebras and the construction of higher indices.
	
	Let $X$ be a proper metric space, i.e. every closed ball in $X$ is compact. An $X$-module is a Hilbert space $H$ equipped with a $*$-representation $\rho\colon C_0(X) \to \mathcal B(H)$ of $C_0(X)$, where $\mathcal B(H)$ is the algebra of all bounded linear operators on $H$. An $X$-module $H$ is called non-degenerate if the $*$-representation  of $C_0(X)$ is non-degenerate, that is, $\rho(C_0(X)) H$ is dense in $H$. An $X$-module is called ample if no nonzero function in $C_0(X)$ acts as a compact operator.
	
	Assume that a discrete group $\Gamma$ acts freely and cocompactly\footnote{More generally, with appropriate modifications, all constructions in this section have their obvious analogues for the case of  proper and cocompact actions instead of free and cocompact actions, cf. \cite[section 2]{MR2732068}.} on $X$ by isometries and $H_X$ is a non-degenerate ample $X$-module equipped with a covariant unitary representation of $\Gamma$. If we denote by $\rho$ and $\pi$ the representations of $C_0(X)$ and $\Gamma$ respectively, this means 
	$$\pi(\gamma)(\rho(f)v)=\rho(\gamma^*f)(\pi(\gamma)v),$$
	where $f\in C_0(X),\gamma\in \Gamma,v\in H_X$ and $\gamma^*f(x)=f(\gamma^{-1}x)$. In this case, we call $(H_X,\Gamma,\rho)$ a covariant system of $(X, \Gamma)$.
	
	\begin{definition}\label{def:fp}
		Let $(H_X,\Gamma,\rho)$ be a covariant system of $(X, \Gamma)$ and $T$ a $\Gamma$-equivariant bounded linear operator acting on $H_X$.
		\begin{enumerate}[(1)]
			\item The propagation of $T$ is defined to be the following supremum
			$$\sup\{\dist(x,y) \mid (x,y)\in \supp(T)\},$$
			where $\supp(T)$ is the complement of points $(x,y)\in X\times X$ for which there exists $f,g\in C_0(X)$ such that $gTf=0$ and $f(x)\ne 0,g(y)\ne 0$;
			\item $T$ is said to be locally compact if $fT$ and $Tf$ are compact for all $f\in C_0(X)$.
		\end{enumerate}
	\end{definition}
	We recall the definition of equivariant Roe algebras. 
	\begin{definition}\label{def:equiroe}
		Let $X$ be a locally compact metric space with a free and cocompact isometric action of $\Gamma$. Let $(H_X,\Gamma,\rho)$ be an covariant system. We define $\mathbb C[X]^\Gamma$ to be the $*$-algebra of $\Gamma$-equivariant locally compact finite propagation operators in $\mathcal B(H_X)$. The equivariant Roe algebra $C^\ast_r(X)^\Gamma$ is defined to be the completion of $\mathbb C[X]^\Gamma$ in $\mathcal B(H_X)$ under the operator norm.
	\end{definition}

	There is also a maximal version of equivariant Roe algebras.
	\begin{definition}
		For an operator $T\in\mathbb{C}[X]^\Gamma$, its \emph{maximal norm} is
		\[\|T\|_{\textnormal{max}}\coloneqq\sup_{\varphi}\left\{\|\varphi(T) \| : \varphi\colon \mathbb{C}[X]^\Gamma\rightarrow\mathcal{B}(H)\textrm{ is a  $*$-representation}\right\}.\]
		The maximal equivariant Roe algebra $C^\ast_{\max}(X)^\Gamma$ is defined to be the completion of $\mathbb{C}[X]^\Gamma$ with respect to $\|\cdot\|_{\textnormal{max}}$.
	\end{definition}
	
	We know 
	\[ C^\ast_r(X)^\Gamma\cong C_r^\ast(\Gamma)\otimes \cpto \textup{ and }  C^\ast_{\max}(X)^\Gamma\cong C_{\max}^\ast(\Gamma)\otimes \cpto,  \]
	where $C_r^\ast(\Gamma)$ (resp. $C^\ast_{\max}(\Gamma)$) is the reduced (resp. maximal) group $C^*$-algebra of $\Gamma$ and $\cpto$ is the algebra of compact operators. 
	
	Furthermore, there are also real versions of reduced and maximal equivariant Roe algebras, by using real Hilbert spaces instead of complex Hilbert spaces. We shall denote these algebras by $C_r^\ast(X)^\Gamma_{\mathbb R}$ and $C_{\max}^\ast(X)^\Gamma_{\mathbb R}$. Similarly, we have 
	\[ C^\ast_r(X)^\Gamma_{\mathbb R} \cong C_r^\ast(\Gamma; \mathbb R)\otimes \cpto_{\mathbb R} \textup{ and }  C^\ast_{\max}(X)^\Gamma_{\mathbb R} \cong C_{\max}^\ast(\Gamma; \mathbb R) \otimes \cpto_{\mathbb R},  \]
	where $C_r^\ast(\Gamma; \mathbb R)$ (resp. $C^\ast_{\max}(\Gamma; \mathbb R)$) is the reduced (resp. maximal) group $C^*$-algebra of $\Gamma$ with real coefficients  and $\cpto_\mathbb R$ is the algebra of compact operators on a real infinite dimensional Hilbert space. 
	
	Let us review the construction of the \emph{higher index} of a first-order symmetric elliptic differential operator on a closed manifold. Suppose $M$ is a closed Riemannian manifold. Let $\widetilde M$ be a Galois covering space of $M$ whose deck transformation group is $\Gamma$.  Suppose  $D$ is a symmetric elliptic differential operator acting on some vector bundle $\spinb$ over $M$. In addition, if $M$ is even dimensional, we assume $\spinb$ to be $\mathbb Z/2$-graded and $D$ has odd-degree with respect to this $\mathbb Z/2$-grading. Let  $\widetilde D$ be the lift of $D$ to $\widetilde M$.  
	
	We choose a noramlizing function $\chi$, i.e.  a continuous odd function $\chi\colon \mathbb R\to \mathbb R$ such that 
	\[ \lim_{x\to \pm \infty} \chi(x) = \pm 1. \] By the standard theory of elliptic operators on complete manifolds, $\widetilde D$ is essentially self-adjoint and $F= \chi(\widetilde D)$ obtained by functional calculus satisfies the condition: 
	\[ F^2 - 1 \in C^\ast_r(\widetilde M)^\Gamma \cong C_r^\ast(\Gamma)\otimes \cpto.\]
	
	In the even dimensional case, since we assume $\spinb$ to be $\mathbb Z/2$-graded and $D$ has odd-degree with respect to this $\mathbb Z/2$-grading, we have
	\[ D = \begin{pmatrix}
		0 & D^-\\
		D^+ & 0 
	\end{pmatrix} \]
	In particular, it follows that   
	\[  F = \begin{pmatrix}
		0 & U \\
		V& 0
	\end{pmatrix}  \]
	for some $U$ and $V$ such that $UV - 1 \in C^\ast_r(\widetilde M)^\Gamma$ and $VU -1\in C^\ast_r(\widetilde M)^\Gamma$.
	Define the following invertible element 
	\[ W \coloneqq \begin{pmatrix} 1 & U \\ 0 & 1\end{pmatrix} \begin{pmatrix} 1 & 0 \\ -V& 1\end{pmatrix} \begin{pmatrix} 1 & U  \\ 0 & 1 \end{pmatrix}\begin{pmatrix} 0 & -1\\ 1 & 0 \end{pmatrix}. \]
	and  form the idempotent 
	\begin{equation}\label{eq:index}
		p = W \begin{pmatrix} 1 & 0 \\ 0 & 0\end{pmatrix} W^{-1} = \begin{pmatrix} UV(2-UV) & (2 - UV)(1-UV) U \\ V(1-UV) & (1-VU)^2\end{pmatrix}.
	\end{equation}
	\begin{definition}\label{def:index}
		In the even dimensional case, the higher index  $\ind_{\Gamma}(\widetilde D)$ of $\widetilde D$ is defined to be
		\[ \ind_{\Gamma}(\widetilde D):= [p] - \left[\begin{psmallmatrix} 1 & 0 \\0 & 0\end{psmallmatrix} \right] \in K_0(C^\ast_r(\widetilde M)^\Gamma) \cong K_0(C_r^\ast(\Gamma)). \]
	\end{definition} 
	Note that if $\Gamma$ is the trivial group, then the higher index $\ind_{\Gamma}(\widetilde D) \in K_0(\cpto) = \mathbb Z$ is simply the classical Fredholm index $\ind(D)$ of $D$, where the latter is defined to be 
	\[  \ind(D) \coloneqq  \dim \textup{ker}(D^+)  - \dim\textup{coker}(D^+). \]

	The construction of higher index in the odd dimensional case is similar. 
	\begin{definition}
		In the odd dimensional case, the higher index  $\ind_{\Gamma}(\widetilde D)$ of $\widetilde D$ is defined to be
		\[ \textstyle \ind_{\Gamma}(\widetilde D):= \exp(2\pi i\frac{\chi(\widetilde D)+ 1}{2}) \in K_1(C^\ast_r(\widetilde M)^\Gamma) \cong K_1(C_r^\ast(\Gamma)). \]
	\end{definition}

	The higher index of $\widetilde D$, as a $K$-theory class, is independent of the choice of the normalizing function $\chi$. In particular, if we choose $\chi$ to be a normalizing function whose distributional Fourier transform has compact support, then $F = \chi(\widetilde D)$ has finite propagation and consequently the formula for defining  $\ind_{\Gamma}(\widetilde D)$ produces an element of finite propagation,\footnote{In the odd dimensional case, one can approximate $\exp(2\pi i\frac{\chi(\widetilde D)+ 1}{2})$ by a finite propagation element, since the coefficients in the power series expansion for the function $e^{2\pi it}$ decays very fast (faster than any exponential decay, to be more precise). } that is, an element in $\mathbb C[\widetilde M]^\Gamma$, which certainly also defines a $K$-theory class in $K_n(C_{\max}^\ast(\Gamma))$. We define this class $\ind_{\Gamma}(\widetilde D)\in K_n(C_{\max}^\ast(\Gamma))$  to be the maximal higher index of the operator $\widetilde D$.

	The higher index of an elliptic operator with real coefficients is defined the same way, and its lies in $KO_{n}(C_r^\ast(\Gamma; \mathbb R))$ or $KO_{n}(C_{\max}^\ast(\Gamma; \mathbb R))$,  when the elliptic operator is appropriately graded (e.g. $\cl_n$-graded with respect to the real Clifford algebra $\cl_n$). See \cite[II. \S 7]{BLMM89}.

	\section{A quantitative relative index theorem}\label{sec:relative} 
	
	In this section, we prove a quantitative relative index theorem (Theorem \ref{thm:relative} and Theorem \ref{thm:maxrelative}), which serves a conceptual framework for studying some conjectures and open questions on Riemannian metrics of positive scalar curvature proposed by Gromov in the past several years \cite{MR3816521, Gromov:2019aa}.

Let $Y$ be a complete $n$-dimensional Riemannian manifold  and $\spinb$ a Euclidean \mbox{$\cl_n$-bundle}  over $Y$.  Suppose $D$ is a first-order symmetric elliptic $\cl_n$-linear differential operator acting on  $\spinb$ over $Y$. Recall the following lemma due to Roe	\cite[Lemma 2.5]{MR3439130}. 
	
	\begin{lemma}\label{lm:fp}
		With the same notation as above,  suppose there exist $\sigma>0$ and a closed subset $K\subset Y$ such that 
		\[  \|Df\| \geq \sigma\|f\| \]
		for all  $f\in C_c^\infty(Y\backslash K, \spinb)$. Given $\varphi \in \mathcal S(\mathbb R)$, assume the Fourier transform $\hat \varphi$ of $\varphi$ is supported in $(-r, r)$. If $\psi\in C_0(Y)$ has support disjoint from the $2r$-neighborhood of $K$, then 
		\[ \|\varphi(D) \rho(\psi) \|_{\op} \leq \|\psi\| \sup\{ |\varphi(y)|\colon |y| \geq \sigma\}. \]
		Here $\rho(\psi)$ is the bounded operator on $L^2(Y, \spinb)$ given by multiplication of $\psi$ and $\|\psi\|$ is supremum norm of $\psi$. 
	\end{lemma}
Roe's proof of the above lemma makes use of the Friedrichs extension of  $D^2$  on $Y\backslash K$ in an essential way. In Appendix \ref{sec:friedrichs}, we shall construct analogues of the Friedrichs extension in the maximal group $C^\ast$-algebra setting, which allows us to extend Roe's lemma above to the corresponding maximal setting (by following essentially the same proof of Roe \cite[Lemma 2.5]{MR3439130}).

 Let us consider the case where $n = \dim Y$ is even. Consider a normalizing function\footnote{A normalizing function is a continuous odd function $\chi\colon \mathbb R\to \mathbb R$ such that $\lim_{x\to \pm \infty} \chi(x) = \pm 1.$} $\chi\colon \mathbb R\to \mathbb R$ such that the distributional Fourier transform $\hat \chi$ of $\chi$ is supported on $[-1,1]$.
Let $F_t = \chi(tD)$. Since $D$ is assumed to have odd-degree with respect to the $\mathbb Z/2$-grading on $\spinb$, we have
\[ D = \begin{pmatrix}
	0 & D^-\\
	D^+ & 0 
\end{pmatrix} \]
In particular, it follows that   
\[  F_t = \begin{pmatrix}
	0 & U_t \\
	V_t& 0
\end{pmatrix}  \]
for some $U_t$ and $V_t$.   As in line \eqref{eq:index}, we define 
\begin{equation}\label{eq:idem}
p_t= \begin{pmatrix} U_tV_t(2-U_tV_t) & (2 - U_tV_t)(1-U_tV_t) U_t \\ V_t(1-U_tV_t) & (1-V_tU_t)^2\end{pmatrix}.
\end{equation}

The following lemma will be useful in the proof of Theorem \ref{thm:relative}. 	
	\begin{lemma}\label{lm:PtD}
	 With the same notation as above, the following hold for all $t>0$. 
		\begin{enumerate}[$(1)$]
			\item There exists $\beta>0$ such that $\|p_{t}\|\leq \beta$ for all $t>0$.
			\item The propagation $\prop(p_{t})$ of $p_t$ is $\leq 5t$.
		\item If there exist $\sigma>0$ and a closed subset $K\subset Y$ such that 
		\[  \|Df\| \geq \sigma\|f\| \]
		for all  $f\in C_c^\infty(Y\backslash K, \spinb)$, then there exists $\eta>0$ such that  
		\[ \|\big(p_{t} -\begin{psmallmatrix}
			1 & 0 \\ 0 & 0 
		\end{psmallmatrix}\big) f\| \leq \frac{\eta}{ t}  \]
	 for all $f\in C_c^\infty(Y\backslash N_{10t}(K), \spinb)$, where $N_{10t}(K)$ is the $10t$-neighborhood of $K$. 
		\end{enumerate}
	\end{lemma}
	\begin{proof}
		Clearly, there exists $\alpha >0$ such that  $|\chi|$ is uniformly bounded by $\alpha$.  Hence both $U_{t}$ and $V_{t}$  have operator norm $\leq \alpha$. Therefore, part (1) follows from the explicit formula of $p_t$ in line \eqref{eq:idem}.
		
		The distributional Fourier transform $\hat \chi$ of $\chi$ is supported on $[-1,1]$. By the inverse Fourier transform formula
		$$\chi(D)=\frac{1}{2\pi }\int \hat{\chi}(\xi) e^{i\xi D}d\xi$$
		and the finite propagation  of the wave operator $e^{i\xi D}$, we see that $\chi(D)$ has propagation no more than $1$. Replacing $\chi(x)$ by $\chi(tx)$, we see that the propagation of $\chi(tD)$ is $\leq t$. In particular,  $U_{t}$ and $V_{t}$  also have propagation $\leq t$. Hence part (2) follows from the explicit formula of $p_t$ in line \eqref{eq:idem}. Furthermore, part (3) follows Lemma \ref{lm:fp}. This finishes the proof. 
	
	\end{proof}

Let $e$ be an element in a Banach algebra such that 
\[\|e^2 - e\|< 1/4.\] 
Then the spectrum of $e$ is disjoint from the vertical line $\{\frac{1}{2} + iy  \mid  y \in \mathbb R\}$. 
Let $\Sigma$ be the part of spectrum of $e$ that is to the right of the line $\{\frac{1}{2} + iy  \mid  y \in \mathbb R\}$. Choose a  contour $\gamma$ containing $\Sigma$ but disjoint from $\Sigma$ and  the vertical line $\{\frac{1}{2} + iy  \mid  y \in \mathbb R\}$. Apply the holomorphic functional calculus and define 
\[    \tilde e = \frac{1}{2\pi i }\int_{\gamma} z (z- e)^{-1}. \]

\begin{lemma}\label{lm:quasi-idem} Let $p$ be an idempotent in a Banach algebra with $\|p\|\leq \beta$. Let $e_s = s p + (1-s)e$ for $s\in [0, 1]$.  Suppose $e$ is an element with $\|e\|\leq \beta+1$ and  $\|p-e\| < (\frac{1}{2\beta+2}) \frac{1}{4}$. Then the following hold. 
	\begin{enumerate}[$(1)$]
		\item We have 
		\[\|e_s^2 - e_s\|< \frac{1}{4}\]
		for all $s\in [0, 1]$. 
		\item If we define 
		\[ \tilde e_s = \frac{1}{2\pi i }\int_{\gamma} z (z- e_s)^{-1} \]
		as above, then $\{e_s\}_{0\leq s\leq 1}$ is a continuous path of idempotents connecting $p$ and $\tilde e_1$.  
	\end{enumerate}	
\end{lemma}
\begin{proof}
	Part (1) follows from the following estimate: 
	\begin{align*}
		\|e^2_s-e_s\| &= \|e^2_s - pe_s + pe_s - p^2 + p- e_s\| \\
		& \leq \|(e_s-p)e_s\| + \|p(e_s-p)\| +\|p-e_s\| \\
		& \leq (\|e_s\| + \|p\| +1) \|e_s-p\|. 
	\end{align*}
By the definition of holomorphic functional calculus, part (2) is obvious.  
\end{proof}

Now we are ready to prove the quantitative relative index theorem. Let us first prove a version of the quantitative relative index theorem for the reduced group $C^\ast$-algebras. The maximal version can be proved in exactly the same way, after applying the results of Appendix \ref{sec:friedrichs}. 

	\begin{theorem}\label{thm:relative}
		Let $Z_1$  and $Z_2$ be two closed $n$-dimensional Riemannian manifold  and $\spinb_j$ a Euclidean $\cl_n$-bundle over $Z_j$ for $j=1, 2$.  Suppose $D_i$ is a first-order symmetric elliptic $\cl_n$-linear differential operators acting on  $\spinb_j$ over $Z_j$. Let $\widetilde Z_j$ be a Galois $\Gamma$-covering space of $Z_j$ and $\widetilde D_j$ the lift of $D_j$. Let $X_j$ be a subset of $Z_j$ and $\widetilde X_j$ the preimage of $X_j$ under the covering map $\widetilde Z_j \to Z_j$. Denote by  $N_r({Z_j\backslash X_j})$ the open $r$-neighborhood of ${Z_j\backslash X_j}$.  Suppose there is $r>0$ such that all geometric data on $N_r({Z_1\backslash X_1})$ and $N_r({Z_2\backslash X_2})$ coincide, i.e. there is an  orientation preserving Riemannian isometry $\Phi\colon N_r({Z_1\backslash X_1}) \to N_r({Z_2\backslash X_2})$ such that $\Phi$ lifts to an isometric $\cl_n$-bundle isomorphism 
		$\Phi\colon \spinb_1|_{N_r({Z_1\backslash X_1})} \to \spinb_2|_{N_r({Z_2\backslash X_2})}$. Assume that 
		\begin{enumerate}[$(1)$]
			\item the restriction of $\widetilde D_j$  on  $\widetilde X_j$ is invertible in the following sense: there exists $\sigma >0$ such that 
			\[  \|\widetilde D_j f\|\geq \sigma \|f\| \]
			for all $f\in C_c^\infty(\interior{\widetilde X}_j, \widetilde {\spinb}_j)$, where $\interior{\widetilde X}_j$ is the interior of $\widetilde X_j$ in $\widetilde Z$;
			\item and $D_1 = \Phi^{-1} D_2 \Phi$ on  $N_r({Z_1\backslash X_1})$. 
		\end{enumerate}
		Then there exists a universal constant $C>0$ such that if  $\sigma\cdot r >C$, then  we have
		\[ \ind_{\Gamma}(\widetilde D_1) - \ind_{\Gamma}(\widetilde D_2) = 0  \] 
		in $KO_{n}(C_{r}^\ast(\Gamma; \mathbb R))$, where  $\ind_{\Gamma}(\widetilde D_j)$ denotes the maximal higher index of $\widetilde D_j$ and   $C^\ast_{r}(\Gamma; \mathbb R)$ is the  reduced  group $C^\ast$-algebra of $\Gamma$ with real coefficients.
	\end{theorem}

	\begin{proof}

		Let us prove the theorem for the case where $\dim Z_j$ is even and $\spinb$ is a Hermitian $\ccl_n$-bundle, mainly for the reason of notational simplicity. Here $\ccl_n$ is the complex Clifford algebra of $\mathbb R^n$. The proof for the real Clifford bundle case is the same. Also, the proof for the odd dimensional case is completely similar.\footnote{Alternatively, the odd dimensional case can be reduced to the even dimensional case by a standard   suspension argument.} Now if $\spinb$ is a Hermitian $\ccl_n$-bundles and $n$ is even, it is equivalent to view  $\spinb$ as a Hermitian vector bundle with a $\mathbb Z/2$-grading, with respect to which  the operators $D_1$ and $D_2$ have odd degree.

We apply the usual higher index construction to $\widetilde D_j$  (cf. Section $\ref{sec:pre}$). Let $\chi\colon \mathbb R \to \mathbb R$ be a normalizing function such that its distributional Fourier transform is supported in $[-1, 1]$. Define 
		\[ F_{1, t} = \chi(t\widetilde D_1) \textup{ and } F_{2, t} = \chi(t\widetilde D_2).   \]
		Let $p_t$ and $q_t$ be the idempotents constructed out of $F_{1, t}$ and $F_{2, t}$ as in line $\eqref{eq:index}$. Then for any fixed $t>0$, the higher index $\ind_\Gamma(\widetilde D_1) \in K_0(C_r^\ast(\Gamma)) $ is represented by 
		\[  [p_t] - \begin{psmallmatrix}
			1 & 0 \\ 0 & 0 
		\end{psmallmatrix}, \] 
	and the higher index $\ind_\Gamma(\widetilde D_2) \in K_0(C_r^\ast(\Gamma)) $ is represented by 
	\[  [q_t] - \begin{psmallmatrix}
		1 & 0 \\ 0 & 0 
	\end{psmallmatrix}. \] 

By assumption, there exists $\sigma >0$ such that 
\[  \|\widetilde D_j f\|\geq \sigma \|f\| \]
for all $f\in C_c^\infty(\interior{\widetilde X}_j, \widetilde {\spinb}_j)$.  By a standard rescaling argument, that is, by considering $\lambda D_1$ and $\lambda D_2$ for some appropriate $\lambda >0$, we can without loss of generality assume $\sigma = 1$, cf. Remark \ref{rm:propspeed}.  	By part (3) of Lemma \ref{lm:PtD}, there exists $\eta>0$ such that 
\[ \|\big(p_{t} -\begin{psmallmatrix}
	1 & 0 \\ 0 & 0 
\end{psmallmatrix}\big) f\| \leq \frac{\eta}{ t}  \]
for all $f\in C_c^\infty(\widetilde Z_1\backslash N_{10t}(\widetilde Z_1\backslash \widetilde X_1), \widetilde \spinb_1)$. Let us define the operator $P_t$ by setting 
\[ P_t(f) = \begin{cases} p_t(f) & \textup{ if } f\in L^2(N_{10t}(\widetilde Z_1\backslash \widetilde X_1), \widetilde \spinb_1) 	\vspace{0.2cm} \\

	\begin{psmallmatrix}
		1 & 0 \\ 0 & 0 
	\end{psmallmatrix}f & \textup{ if }  f\in L^2(\widetilde Z_1\backslash N_{10t}(\widetilde Z_1\backslash \widetilde X_1), \widetilde \spinb_1).
\end{cases} \] In particular, by Lemma \ref{lm:quasi-idem}, as long as $t = t_0 $ is sufficiently large, then $\|P_t -p_t\|$ is sufficiently small, which implies that 
\[  \|P_t^2 - P_t\| < \frac{1}{4}. \]
Furthermore, if we define 
		\[ \widetilde P_t = \frac{1}{2\pi i }\int_{\gamma} z (z- P_t)^{-1} \]
then part (2) of Lemma \ref{lm:quasi-idem} implies that $\widetilde P_t$ and $p_t$ represent the same $K$-theory class, as long as $t = t_0$ is sufficiently large. 

We apply the same argument above to $q_t$ and define 
\[ Q_t(f) = \begin{cases} q_t(f) & \textup{ if } f\in L^2(N_{10t}(\widetilde Z_2\backslash \widetilde X_2), \widetilde \spinb_2) 	\vspace{0.2cm} \\
	
	\begin{psmallmatrix}
		1 & 0 \\ 0 & 0 
	\end{psmallmatrix}f & \textup{ if }  f\in L^2(\widetilde Z_2\backslash N_{10t}(\widetilde Z_2\backslash \widetilde X_2), \widetilde \spinb_2).
\end{cases} \]
Similarly, we define 	\[ \widetilde Q_t = \frac{1}{2\pi i }\int_{\gamma} z (z- Q_t)^{-1} \]
Then Lemma \ref{lm:quasi-idem} implies that $\widetilde Q_t$ and $q_t$ represent the same $K$-theory class, as long as $t = t_0$ is sufficiently large. 

Now let us set $C = 15t_0$. Recall that we have already applied a rescaling argument to reduce the general case to the case where $\sigma =1$. Then by assumption, we have 
\[ r = \sigma \cdot r > C = 15t_0. \]  It follows from the standard finite propagation of wave operators associated to $\widetilde D_1$ and $\widetilde D_2$ that $p_{t_0}$ and $q_{t_0}$ coincide as operators\footnote{As far as $K$-theory classes of $p_{t_0}$ and $q_{t_0}$ are concerned, we can simply ignore the subspaces $L^2(\widetilde Z_1\backslash N_{r}(\widetilde Z_1\backslash \widetilde X_1), \widetilde \spinb_1)$ and  $L^2(\widetilde Z_2\backslash N_{r}(\widetilde Z_2\backslash \widetilde X_2), \widetilde \spinb_2)$, since $p_{t_0}$ and $q_{t_0}$ act on them as the trivial idempotent $\begin{psmallmatrix}
		1 & 0 \\ 0 & 0 
	\end{psmallmatrix}$ respectively. } on 
\[ L^2(N_{r}(\widetilde Z_1\backslash \widetilde X_1), \widetilde\spinb_1) \cong L^2(N_{r}(\widetilde Z_2\backslash \widetilde X_2), \widetilde\spinb_2).\]
It follows that $\widetilde P_{t_0}$ and $\widetilde Q_{t_0}$ coincide. In particular, we have 
\[  \ind_\Gamma(\widetilde D_1)  = [\widetilde P_{t_0}] - [\begin{psmallmatrix}
	1 & 0 \\ 0 & 0 
\end{psmallmatrix}] =  [\widetilde Q_{t_0}] - [\begin{psmallmatrix}
1 & 0 \\ 0 & 0 
\end{psmallmatrix}]  = \ind_\Gamma(\widetilde D_2) \]
in $K_0(C_r^\ast(\Gamma; \mathbb R))$. This finishes the proof.

	\end{proof}

By applying the results of Appendix \ref{sec:friedrichs} and the estimates in Example \ref{ex:dirac}, the same proof for Theorem \ref{thm:relative} also proves the following maximal version of the quantitative relative index theorem. We omit the details. 

\begin{theorem}[Theorem \ref{thm:relative-intro}]\label{thm:maxrelative}
	Let $Z_1$  and $Z_2$ be two closed $n$-dimensional Riemannian manifold  and $\spinb_j$ a Euclidean $\cl_n$-bundle over $Z_j$ for $j=1, 2$.  Suppose $D_j$ is a  $\cl_n$-linear Dirac-type operator acting on  $\spinb_j$ over $Z_j$. Let $\widetilde Z_j$ be a Galois $\Gamma$-covering space of $Z_j$ and $\widetilde D_j$ the lift of $D_j$. Let $X_j$ be a subset of $Z_j$ and $\widetilde X_j$ the preimage of $X_j$ under the covering map $\widetilde Z_j \to Z_j$. Denote by  $N_r({Z_j\backslash X_j})$ the open $r$-neighborhood of ${Z_j\backslash X_j}$.  Suppose there is $r>0$ such that all geometric data on $N_r({Z_1\backslash X_1})$ and $N_r({Z_2\backslash X_2})$ coincide, i.e. there is an  orientation preserving Riemannian isometry $\Phi\colon N_r({Z_1\backslash X_1}) \to N_r({Z_2\backslash X_2})$ such that $\Phi$ lifts to an isometric $\cl_n$-bundle isomorphism 
	$\Phi\colon \spinb_1|_{N_r({Z_1\backslash X_1})} \to \spinb_2|_{N_r({Z_2\backslash X_2})}$. Assume that 
	\begin{enumerate}[$(1)$]
		\item  there exists $\sigma>0$ such that 
		\[ \mathcal R_j(x) \geq \frac{(n-1)\sigma^2}{n} \] for all $x\in X_j$, where $\mathcal R_j$ is the curvature term appearing in $D_j^2 = \nabla^\ast \nabla + \mathcal R_j$, 
		\item and $D_1 = \Phi^{-1} D_2 \Phi$ on  $N_r({Z_1\backslash X_1})$. 
	\end{enumerate}
	Then there exists a universal constant $C>0$ such that if  $\sigma\cdot r > C$, then  we have
	\[ \ind_{\Gamma}(\widetilde D_1) - \ind_{\Gamma}(\widetilde D_2) = 0  \] 
	in $KO_{n}(C_{\max}^\ast(\Gamma; \mathbb R))$, where  $\ind_{\Gamma}(\widetilde D_j)$ denotes the maximal higher index of $\widetilde D_j$ and   $C^\ast_{\max}(\Gamma; \mathbb R)$ is the  maximal group $C^\ast$-algebra of $\Gamma$ with real coefficients.
\end{theorem}
		The numerical estimates  in Appendix \ref{app:estimate} show that the universal constant $C$ is $\leq 40.65.$
	
\begin{remark}\label{rm:propspeed}
	In the proof of Theorem \ref{thm:relative}, we have implicitly used the fact a Dirac-type operator has propagation speed equal to $1$. Recall that the propagation speed of   a first order differential operator $D$ on a Riemannian manifold $Z$ is defined to be   
	\[ c_D  \coloneqq \sup_{x\in X} c_D(x),  \] 
	where $\sigma_D$ the principal symbol of $D$ and 
		\[ c_D(x)  \coloneqq \sup\{ \|\sigma_D(x, \xi)\|  :  \xi\in T_x^\ast X, \|\xi\| = 1\}.  \]
	If we consider more general elliptic operators $D_1$ and $D_2$ such that both $c_{D_1}$ and $c_{D_2}$ are bounded by $\lambda$, then the corresponding condition $\sigma\cdot r >C$ in Theorem \ref{thm:relative} should be replaced by 
	\[ \sigma \cdot \frac{r}{\lambda }>C. \] 
		 
\end{remark}

The following is a typical geometric setup to which Theorem $\ref{thm:maxrelative}$ applies.

\begin{example}\label{ex:dirac}
	Let $Z$ be a closed $n$-dimensional Riemannian manifold  and $\spinb$ a Euclidean $\cl_n$-bundle over $Z$.  Suppose $D$ is a  $\cl_n$-linear Dirac-type operator acting on  $\spinb$ over $Z$. Let $\widetilde Z$ be a Galois $\Gamma$-covering space of $Z$ and $\widetilde D$ the lift of $D$. Let $X$ be a subset of $Z$ and $\widetilde X$ the preimage of $X$ under the covering map $\widetilde Z \to Z$. 
	
	By the Lichnerowicz formula, we have 
	\[ \widetilde D^2 = \nabla^\ast\nabla + \mathcal R, \]
	where $\mathcal R$ is a symmetric bundle endomorphism of $\widetilde \spinb$. If $D$ is an actual Dirac operator, then $\mathcal R = \frac{\kappa}{4}$ where $\kappa$ is the scalar curvature of the metric on $\widetilde Z$. 
	By the Cauchy–Schwarz inequality, we have
	\[  \langle \widetilde D f, \widetilde D f\rangle \leq n \langle \nabla f, \nabla f\rangle   \]
	for all $f\in C_c^\infty({\widetilde Z}, \widetilde \spinb)$ and  $n = \dim Z$. Combining the two formulas above,  we see that  
	\begin{equation*}
		\frac{n-1}{n}\langle \widetilde Df, \widetilde D f\rangle \geq  \langle \mathcal Rf, f\rangle 
	\end{equation*} 
	for all $f\in C_c^\infty({\widetilde Z}, \widetilde \spinb)$. 
	
	Let $\interior{\widetilde X} = \widetilde X - \partial \widetilde X$ be the interior of $\widetilde X$. If we assume there exists $\sigma >0$ such that 
	\[  \mathcal R(x) \geq  \frac{(n-1)\sigma^2}{n}  \]
	for all $x\in \interior{\widetilde X}$,     then we have
	\begin{equation}\label{eq:lichbd}
		\frac{n-1}{n}\langle \widetilde Df, \widetilde D f\rangle \geq  \langle \mathcal Rf, f\rangle \geq \sigma^2 \langle f, f\rangle 
	\end{equation} 
	for all $f\in C_c^\infty(\interior{\widetilde X}, \widetilde \spinb)$. In other words,   we have  
	\begin{equation}\label{eq:lowerbd}
		\| \widetilde D f\| \geq \sigma\|f\|
	\end{equation} 
	for all $f\in C_c^\infty(\interior{\widetilde X}, \widetilde \spinb)$ in this case.

\end{example}

	\section{Proof of Theorem \ref{thm:gcube} }\label{sec:gcube}

	In this section, we apply the quantitative relative index theorem (Theorem \ref{thm:maxrelative}) to prove Theorem \ref{thm:gcube}.  In order to  make our exposition more transparent,  let us first prove the following special case.

	\begin{theorem}[A special case of Theorem \ref{thm:gcube}]\label{thm:band2}
		
		Suppose $M$ is a closed spin manifold of dimension $n-1$ such that the higher index of its Dirac operator   does not vanish in $KO_{n-1}(C^\ast_{\max}(\pi_1 M; \mathbb R))$.  If the manifold $M \times [0, 1]$ is endowed with  a Riemannian metric whose scalar curvature is $\geq n(n-1)$, then 
		\[  \width(M\times [0, 1]) \leq \frac{\frac{8}{\sqrt{3}}C + 4\pi}{n}. \] 
		where $C$ is the universal constant from Theorem $\ref{thm:maxrelative}$.
	\end{theorem}

	\begin{proof}
		For simplicity, we shall prove the theorem for the reduced case. In fact,  let us assume that the higher index of the (complexified) Dirac operator  on $M$ does not vanish in $K_{n-1}(C_r^\ast(\Gamma))$. The proof for the maximal case is essentially the same. For the real case, see Remark \ref{rm:real}.

		Let $\widetilde X = \widetilde M \times [0, 1]$ be the universal cover of $X$ and $\widetilde D$ the associated $\ccl_n$-linear Dirac operator on  $\widetilde X$. By the discussion in Example $\ref{ex:dirac}$, since the scalar curvature $\Sc(g)\geq n(n-1)$, we have 
		\[   \| \widetilde Df\| \geq\frac{n}{2} \|f\| \]
		for all  $f\in C_c^\infty(\interior{\widetilde X}, \widetilde\spinb)$, where $\widetilde \spinb$ is the associated spinor bundle over $\widetilde X$.

		We prove the theorem by contradiction. Assume to the contrary that 
		\[  \ell \coloneqq  width(X) > \frac{\frac{8}{\sqrt{3}}C + 4\pi}{n}. \]
		Denote by $\partial_+X = M\times \{1\}$ and $\partial_-X = M\times \{0\}$. Then  for any sufficiently small $\varepsilon > 0$, there exists a hypersurface $Y$ in $X$ such that 
		\[ \dist(\partial_+X, Y) \geq \frac{\ell}{2} - \varepsilon \textup{ and } \dist(\partial_-X, Y) \geq \frac{\ell}{2} - \varepsilon.  \]  
		Let $\varphi\colon X\to \mathbb R$ be a real-valued smooth function such that (cf. \cite[proposition 2.1]{MR532376})
		\begin{enumerate}[$(1)$]
			\item $\|d\varphi\| \leq  1$,
			\item $\varphi(x) \equiv 0$ for all $x$ between $Y$ and $\partial_-X$ with  $\dist(x, Y) \geq \frac{2\pi}{n} + \varepsilon $, 
			\item and  $\varphi(x) \equiv \frac{4\pi}{n}$ for all $x$ between $Y$ and $\partial_+X$ with  $\dist(x, Y) \geq \frac{2\pi}{n} + \varepsilon$.
		\end{enumerate}
		From now on,  let us fix a sufficiently small $\varepsilon > 0 $ and denote the lift of $\varphi$ from $X$ to $\widetilde X$ still by $\varphi$. Define the function
		\[   u(x) = e^{
			\frac{n}{2} i\varphi(x) } \]
		on $\widetilde X$. We have  
		\[ \|[\widetilde D, u]\| = \|du\|  = \frac{n}{2}\|u\cdot d\varphi\| \leq \frac{n}{2}.  \]
		Similarly, we also have 
		\[  \|[\widetilde D, u^{-1}]\|\leq  \frac{n}{2}.  \]
		
		Consider the following Dirac operator on $\sone\times \interior{\widetilde X}$:
		\begin{equation}\label{eq:diracflow}
			\slashed{D} = c \cdot \frac{d}{d t} +  \widetilde D_t
		\end{equation}
		where $c$ is the Clifford multiplication of the unit vector $d/dt$ and 
		\[ \widetilde D_t \coloneqq t\widetilde D + (1-t) u\widetilde Du^{-1} \] 
		for each $t\in [0, 1]$. Here we have chosen the parametrization $\sone= [0, 1]/\{0, 1\}$. Let $\widetilde \spinb_{[0,1]}$  be the associated spinor bundle on $[0, 1]\times \interior{\widetilde X}$ and $\widetilde S_t$  its restriction on $\{t\} \times \interior{\widetilde X}$. Each smooth section $f\in C^\infty_c([0, 1]\times \interior{\widetilde X}, \widetilde \spinb_{[0,1]})$ can be viewed as a smooth family $f(t) \in C_c^\infty(\{t\}\times \interior{\widetilde X}, \widetilde \spinb_t)$. The operator $\slashed{D}$ acts on the following subspace  of  $C^\infty_c([0, 1]\times \widetilde X, \widetilde \spinb_{[0,1]})$:
		\[  \{f\in C^\infty_c([0, 1]\times \interior{\widetilde X}, \widetilde \spinb_{[0,1]}) \mid f(1) = uf(0)\}.  \]
		From now on, we shall simply write  $C^\infty_c(\sone\times \interior{\widetilde X}, \widetilde \spinb)$ for the above subspace of sections. 
		
		Clearly, we have 
		\[ \slashed{D}^2  = -\frac{d^2}{dt^2} + \widetilde D_t^2  + c [\widetilde D, u]u^{-1}. \]
		By using the identity 
		\[ \widetilde Du \widetilde Du^{-1} + u\widetilde  Du^{-1}\widetilde D= [\widetilde D, u][\widetilde D, u^{-1}] + u\widetilde D^2u^{-1} + \widetilde  D^2, \]
		we have
		\begin{equation}\label{eq:diracsquare}
\widetilde D_t^2 = t \widetilde D^2 + (1-t)u\widetilde D^2u^{-1} + t(1-t) [\widetilde D, u][\widetilde D,u^{-1}]. 
		\end{equation} 
		It follows from the assumption $\Sc(g)\geq n(n-1)$ and the estimates in Example \ref{ex:dirac} that  $ \widetilde D^2 \geq \frac{n^2}{4}$ on  $C^\infty_c(\sone\times \interior{\widetilde X}, \widetilde \spinb)$, which implies  also 
		${u\widetilde  D^2 u^{-1} \geq \frac{n^2}{4}}$ on $C^\infty_c(\sone\times \interior{\widetilde X}, \widetilde \spinb)$, since $u$ is a unitary.  Therefore, we have 
		\begin{align*}
			\widetilde D_t^2 & \geq \frac{n^2}{4}  - t(1-t) \|[\widetilde D, u^{-1}][\widetilde D,u] \|  \\
			& \geq \frac{n^2}{4}  - \frac{n^2}{16}  = \frac{3n^2}{16}
		\end{align*}
		where the second inequality uses the fact $t(1-t) \leq 1/4$ for all $t\in [0, 1]$.

		Now for each $\lambda >0$, we define the rescaled version of $\slashed D$ to be  
		\begin{equation}\label{eq:rescale}
			\slashed D_\lambda = c \cdot \frac{d}{dt} + \lambda  \widetilde D_t 
		\end{equation}
		with $\lambda \widetilde D_t$ in place of $\widetilde D_t$. The same calculation from above shows that 
		\[ \slashed D_\lambda^2 =  -\frac{d^2}{dt^2} + \lambda^2 \widetilde D_t^2  + \lambda c [\widetilde D, u]u^{-1}. \]
		Since $\widetilde D_t^2 \geq \frac{3n^2}{16}$, it follows that 
		\[ \slashed D_\lambda^2 \geq \lambda^2 \frac{3n^2}{16} - \lambda \frac{n}{2} = \frac{3n^2 \lambda^2 }{16}\big(1 - \frac{8}{3n\lambda}\big).   \]
		 If we want to be explicit about the dependence of $\slashed{D}_\lambda$ on the unitary $u$, we shall write $\slashed{D}_{\lambda, u}$ instead of $\slashed{D}_\lambda$. 
		
		Let $v \equiv 1$ be the trivial unitary on $\widetilde X$. Define the operator 
		\[ \slashed{D}_{\lambda, v} = c \frac{d}{dt} +  \lambda \widetilde D.\] 
		A similar (in fact simpler) calculation shows that 
	\[ \slashed D_{\lambda, v}^2 \geq \lambda^2 \frac{n^2}{4}  \]
		on  $C^\infty_c(\sone\times \interior{\widetilde X}, \widetilde \spinb)$.

		Consider the doubling $\double{X} = M\times \sphere^1 = X\cup_{\partial X} (-X)$ of $X$, where $-X$ is a copy of $X$ but with the opposite orientation. Extend\footnote{To be precise, we fix a copy of $X$ inside of $\double X$ and equip it with the  Riemannian metric given by the assumption. Then we choose any Riemannian metric on $\double{X}$ that coincides with the Riemannian metric on this chosen copy of $X$.}  the Riemannian metric on $X$ to a Riemannian metric on $\double{X}$. 
		The reader should not confuse the copy of $\sone$ appearing in $\double{X} = M\times \sone$ with the copy of $\sone$ appearing in $\sone \times \interior{X} = \sone \times M\times (0,1)$.  Note that the Riemannian metric on $\double{X} = M\times \sone$ does \emph{not} have positive scalar curvature everywhere in general. But $\double{X}$ is a closed manifold, so the usual higher index theory applies. More precisely, by the construction of $u = e^{\frac{n}{2} i\varphi}$, it extends trivially  to a unitary $\double{u}$ on $\widetilde{\double{X}}\coloneqq \widetilde M \times \sone$ by setting it to be $1$ in $\widetilde{\double{X}}\backslash \widetilde X$.  Let $\widetilde D^{\double{X}}$ be the Dirac operator on $ \widetilde{\double{X}}$. We define 
		\[  \slashed{D}^{\double X}_{\double{u}} = c\cdot \frac{d}{dt} + \widetilde D^{\double X}_t  \textup{ where } \widetilde D^{\double{X}}_t \coloneqq t \widetilde D^{\double{X}} + (1-t) \double{u} \widetilde D^{\double{X}} \double{u}^{-1}.  \]
		Similarly, let $\double{v}\equiv 1$ be the trivial unitary on $\widetilde{\double{X}}$ and define 
		\[ \slashed{D}^{\double{X}}_{\double{v}} = c\cdot \frac{d}{dt} + \widetilde D^{\double{X}}.\]
		\begin{claim*}
			$\ind_{\Gamma}(\slashed{D}^{\double{X}}_{\double{u}}) = \ind_{\Gamma}(\widetilde D^M)$ in $K_{n-1}(C_r^\ast(\Gamma)),$ where $\Gamma = \pi_1 M$ and $\widetilde D^M$ is the Dirac operator on $\widetilde M$. 
		\end{claim*}
		This can  be seen as follows. The higher index $\ind_{\Gamma}(\slashed{D}^{\double{X}}_{\double{u}}) $ is independent of the choice of the Riemannian metric on $\double{X}$,  since $\double{X} = M\times \sphere^1$ is a closed manifold. Furthermore, if $\{\double{u}_s\}_{0\leq s \leq 1}$ is a continuous family of unitaries on $\double{X}$, then $\ind_{\Gamma}(\slashed{D}^{\double{X}}_{\double{u}_0}) = \ind_{\Gamma}(\slashed{D}^{\double{X}}_{\double{u}_1}) \in K_{n-1}(C_{r}^\ast(\Gamma))$. Therefore, without loss of generality,  we assume the Riemannian metric on $\double{X} = M\times \sphere^1$ is given by a product metric $g_M + dx^{2}$ and assume\footnote{This can be achieved by a homotopy of unitaries on $\double X$.} the unitary $\double{u}$ on $\double{X}$ is given by the projection map $\double{X} = M\times \sphere^1  \to \sphere^1 \subset \mathbb C$. In this case, the operator  $\slashed{D}^{\double{X}}_{\double{u}}$ becomes 
		\begin{equation*}
			\big(c \frac{d}{d t} + D_t^{\sphere^1}\big)\hotimes 1 + 1\hotimes \widetilde D_M
		\end{equation*} 
		where $ D_t^{\sphere^1} = t D^{\sphere^1} + (1-t) e^{2\pi i \theta} D^{\sphere^1}e^{-2\pi i \theta}$ and $\theta$ is the coordinate for the copy of $\sphere^1$ appearing in $\double X = M\times \sphere^1$.  Recall that the index of the operator $c \frac{d}{dt} + D^{\sphere^1}_t$ is equal to the spectral flow of the family $\{D^{\sphere^1}_t\}_{0\leq t \leq 1}$, which has index $1$ (cf. \cite[Section 7]{A-P-S76}). Therefore, it follows that 
		\[  \ind_{\Gamma}(\slashed{D}^{\double{X}}_{\double{u}}) = \ind_{\Gamma}(\widetilde D^M) \]
		in $ K_{n-1}(C_{r}^\ast(\Gamma))$.
		The same argument also shows that  
		\[  \ind_{\Gamma}(\slashed{D}^{\double{X}}_{\double{v}}) = 0  \textup{ in }  K_{n-1}(C_{r}^\ast(\Gamma)). \]
		We conclude that 
		\[\ind_{\Gamma}(\slashed{D}^{\double{X}}_{\double{u}}) - \ind_{\Gamma}(\slashed{D}^{\double{X}}_{\double{v}})  = \ind_{\Gamma}(\widetilde D^M) \]
		in $K_{n-1}(C_{r}^\ast(\Gamma))$.
		
			On the other hand, since we have assumed  that
		\[  width(X) > \frac{\frac{8}{\sqrt{3}}C + 4\pi}{n}, \]
		the operators $\slashed{D}^{\double{X}}_{\double{u}}$ and $\slashed{D}^{\double{X}}_{\double{v}}$ coincide on the $r$-neighborhood $N_{r}(\double X\backslash X)$ of  $\double X\backslash X$, where we have 
		\[ r >\frac{4C}{\sqrt{3}\,n} \]
		as long as $\varepsilon$  chosen at the beginning of the proof is sufficiently small. In particular, we see that 
		\[ \frac{r}{\lambda}\sqrt{\frac{3n^2 \lambda^2 }{16}\big(1 - \frac{8}{3n\lambda}\big) } > C, \]
		as long as $\lambda$ is sufficiently large. 
		Now it follows from Theorem $\ref{thm:maxrelative}$ and Remark \ref{rm:propspeed} that  
		\[\ind_{\Gamma}(\slashed{D}^{\double{X}}_{\double{u}}) - \ind_{\Gamma}(\slashed{D}^{\double{X}}_{\double{v}})   =  0  \]
		in $K_{n-1}(C_{r}^\ast(\Gamma))$.   We arrive at a contradiction, since $\ind_{\Gamma}(\widetilde D^M) \neq 0$ by assumption. This finishes the proof.

	\end{proof}

	\begin{remark}\label{rm:real}
		Let us discuss how to adjust the proof of Theorem \ref{thm:band2} for the real case. Roughly speaking, we replace the imaginary number $ i = \sqrt{-1}$ by the matrix $ \bm{I} = \begin{psmallmatrix}
			0 & 1 \\
			-1 & 0 
		\end{psmallmatrix},$
		while viewing $\bm{I}$ as a matrix acting on a $2$-dimensional $\mathbb Z/2$-graded real vector space.  For example,  multiplication by the complex number $e^{2\pi i t}$ on a $1$-dimensional complex vector space  is replaced by the operator $ e^{2\pi t\cdot \bm{I} } $
		acting on a $2$-dimensional $\mathbb Z/2$-graded real vector space. More precisely,  let us describe such a modification in terms of Clifford algebras. Let $\cl_{r, s}$ be the real Clifford algebra generated by $\{e_1, e_2, \cdots, e_{r+s}\}$ subject to the following relations: 
		\[ e_je_k + e_ke_j = \begin{cases}
			- 2\delta_{jk} & \textup{ if } j \leq r \\
			+ 2 \delta_{jk}  & \textup{ if } j > r. 
		\end{cases} \]
		To be clear,  our convention for the notation of Clifford algebras is consistent with that of  \cite{BLMM89}.  In particular,  $\cl_n \coloneqq \cl_{n, 0}$ stands for the Clifford algebra generated by 
		by $\{e_1, e_2, \cdots, e_n\}$ subject to the following relations: 
		\[ e_j^2 = - 1 \textup{ and } e_je_k + e_je_k = 0 \textup{ for all } 1\leq j, k \leq n. \]
		In terms of Clifford algebras, we define $\bm{I} = e_1e_2$ in  $\cl_{0, 2}$.  The operator $\slashed{D}$ in line $\eqref{eq:diracflow}$ now becomes 
		\[   \slashed D = c \cdot \frac{d}{dt} + \widetilde D_t,  \]
		where $c\in \cl_{1, 0}$ is the Clifford multiplication of the unit vector $d/dt$ and 
		\[ \widetilde D_t \coloneqq t\widetilde D + (1-t) \bm{U} \widetilde D\bm{U}^{-1} \] 
		with $\bm{U} = e^{2\pi t \bm{I} \varphi(x)/\ell}.$ In particular, the operator $\slashed D$ is a $\cl_{ n+1, 2}$-linear Dirac-type operator and its higher index lies in $ KO_{n-1}(C_{\max}^\ast(\Gamma; \mathbb R)). $ The same remark applies to other similar operators that appeared in the proof of Theorem \ref{thm:band2}. With these modifications,  the proof for the real case now proceeds in the same way as the complex case. 
	\end{remark}

	Now we are ready to prove Theorem \ref{thm:gcube}.

	\begin{theorem}[Theorem \ref{thm:gcube}]\label{thm:gcube2}
		Let $X$ be an $n$-dimensional  compact connected spin manifold with boundary.
		Suppose 
		$f\colon X\to [-1, 1]^m$
		is  a smooth map that  sends the boundary of $X$ to the boundary of $[-1, 1]^m$. 
		Let $\partial_{j\pm}, j = 1, \dots, m$,  be the pullbacks of the pairs of the opposite faces of the cube $[-1, 1]^m$. 
		Suppose  $Y_{\pitchfork}$ is an $(n-m)$-dimensional closed submanifold \textup{(}without boundary\textup{)} in $X$ that satisfies the following conditions:
		\begin{enumerate}[$(1)$]
			\item $\iota \colon \pi_1(Y_\pitchfork) \to \pi_1(X)$ is injective, where $\iota$ is the canonical morphism on $\pi_1$ induced by the inclusion $Y_\pitchfork\hookrightarrow \pi_1(X)$;
			\item $Y_{\pitchfork}$  is the transversal intersection  of $m$ orientable hypersurfaces $Y_j\subset X$, $1\leq j \leq m$,  such that each $Y_j$  separates $\partial_{j-}$ from $\partial_{j+}$;
			\item the higher index $\ind_{\Gamma}(D_{Y_\pitchfork})$ does not vanish  in $KO_{n-m}(C^\ast_{\max}(\Gamma; \mathbb R))$, where ${\Gamma = \pi_1(Y_\pitchfork)}$.  
		\end{enumerate}  If $\Sc(X) \geq  n(n-1)$, then  the distances $\ell_j = \dist(\partial_{j-}, \partial_{j+})$ satisfy the following inequality:
		\[ \sum_{j=1}^m \frac{1}{\ell_j^2} \geq \frac{n^2}{
	(\frac{8}{\sqrt{3}}C + 4\pi)^2 }. \]
		Consequently, we have 
		\[  \min_{1\leq j\leq m} \dist(\partial_{j-}, \partial_{j+}) \leq \sqrt{m} \frac{\frac{8}{\sqrt{3}}C + 4\pi}{n}.\]
	\end{theorem}
	\begin{proof}
		
		For simplicity, we shall prove the theorem for the complex case, that is, complexified Dirac operators instead of $\cl_n$-linear Dirac operators. For the real case, see Remark \ref{rm:real}.

		We first show that the general case where $\iota\colon \pi_1(Y_\pitchfork) \to \pi_1(X)$ is injective  can be reduced to   the case where $\iota\colon \pi_1(Y_\pitchfork) \to \pi_1(X)$ is split injective.\footnote{We say $\iota\colon \pi_1(Y_\pitchfork) \to \pi_1(X)$ is split injective if there exists a group homomorphism $ \varpi\colon \pi_1(X) \to \pi_1(Y_\pitchfork) $ such that $\varpi\circ \iota = \id$, where $\id$ is the identity morphism of $\pi_1(Y_\pitchfork)$.} Let $X_{u}$ be the universal cover of $X$. Since by assumption $\iota\colon \pi_1(Y_\pitchfork) \to \pi_1(X)$ is injective, we can view $\Gamma = \pi_1(Y_\pitchfork)$ as a subgroup of $\pi_1(X)$.  Let $X_\Gamma = X_u/ \Gamma$ be the covering space of $X$ corresponding to the subgroup $\Gamma\subset \pi_1(X)$. Then the inverse image of $Y_\pitchfork$ under the projection $p\colon X_\Gamma \to X$ is a disjoint union of covering spaces of $Y_\pitchfork$, at least one of which is a diffeomorphic copy of $Y_\pitchfork$. Fix such a copy of $Y_\pitchfork$ in $X_\Gamma$ and denote it by $\widehat Y_\pitchfork$.  Roughly speaking, the space $X_\Gamma$ equipped with the lifted Riemannian metric from $X$ could serve as a replacement of the original space $X$, except that $X_\Gamma$ is not compact in general. To remedy this, we shall choose a ``fundamental domain" around $\widehat Y_\pitchfork$ in $X_\Gamma$ as follows. 
		
		By assumption, $Y_\pitchfork \subset X$ is the transversal intersection of $m$ orientable hypersurfaces $Y_j \subset X$. Let $r_j$ be the distance function\footnote{To be precise, let $r_j$ be a smooth approximation of the distance function from $\partial_{j-}$.} from $\partial_{j-}$, that is $r_j(x) = \dist(x, \partial_{j-})$. Without loss of generality,  we can assume $Y_j = r_j^{-1}(a_j)$ for some regular value $a_j\in [0, \ell_j]$.  Let $Y^\Gamma_j = p^{-1}(Y_j)$ be the inverse image of $Y_j$ in $X_\Gamma$. Denote by $\overbar{r}_j$ the lift of $r_j$ from $X$ to $X_\Gamma$. Let $\nabla \overbar{r}_j$ be the gradient vector field associated to $\overbar{r}_j$. A point $x\in X_\Gamma$ said to be \emph{permissible} if there exist a number $s\geq 0$ and a piecewise smooth curve $c\colon [0, s] \to X_\Gamma$ satisfying the following conditions: 
		\begin{enumerate}[(i)]
			\item $c(0) \in \widehat Y_{\pitchfork}$ and $c(s) = x$; 
			\item there is a subdivision of $[0, s]$ into finitely many subintervals $\{[t_{k}, t_{k+1}]\}$ such that,  on each subinterval $[t_{k}, t_{k+1}]$, the curve $c$ is either an integral curve or a reversed integral curve\footnote{By definition, an integral curve of a vector field is a curve whose tangent vector coincides with the given vector field at every point of the curve. A reversed integral curve is an integral curve with the reversed parametrization, that is,  the tangent vector field of a reserved integral curve coincides  with the negative of the given vector field at every point of the curve. } of the gradient vector field $\nabla \overbar{r}_{i_k}$ for some $1\leq i_k\leq m$, where we require $i_{k}$'s to be all distinct from each other;
			\item furthermore, when $c$ is an integral curve of the gradient vector field $\nabla \overbar{r}_{i_k}$ on the subinterval $[t_{k}, t_{k+1}]$, we require the length of $c|_{[t_{k}, t_{k+1}]}$ to be less than or equal to $(\ell_{i_k} - a_{i_k} -\frac{\varepsilon}{4})$; and when $c$ is a reversed integral curve of the gradient vector field $\nabla \overbar{r}_{i_k}$ on the subinterval $[t_{k}, t_{k+1}]$, we require the length of $c|_{[t_{k}, t_{k+1}]}$ to be less than or equal to $(a_{i_k} -\frac{\varepsilon}{4})$.
		\end{enumerate} 
		
		Let $T$ be the set of all permissible  points. Now $T$ may not be a manifold with corners. To fix this, we choose an open cover $\mathscr U = \{U_\alpha\}_{\alpha\in \Lambda}$ of $T$ by geodesically convex metric balls of sufficiently small radius $\delta > 0 $. Now take the union of members of $\mathscr U = \{U_\alpha\}_{\alpha\in \Lambda}$ that do not intersect the boundary $\partial T$ of $T$, and  denote by $Z$ the closure of the resulting subset. Then $Z$ is a manifold with corners which,   together with the subspace $\widehat Y_\pitchfork\subset Z$,  satisfies all the conditions of the theorem, provided that $	\varepsilon$ and $\delta$ are chosen to be sufficiently small. In particular, the intersection $Y_j^\Gamma\cap Z$ of each hypersurface $Y_j^\Gamma$ with $Z$ gives a hypersurface of $Z$.  The transerval intersection of the resulting hypersurfaces is precisely $\widehat Y_\pitchfork \subset Z$.   Furthermore, note that the isomorphism $\Gamma = \pi_1(Y^\Gamma_\pitchfork) \to \pi_1(X^\Gamma) = \Gamma$ factors as   the composition $\pi_1(Y^\Gamma_\pitchfork) \to \pi_1(Z) \to \pi_1(X^\Gamma)$, where the morphisms $\pi_1(Y^\Gamma_\pitchfork) \to \pi_1(Z)$ and  $\pi_1(Z) \to \pi_1(X^\Gamma)$ are induced by the obvious inclusions of spaces. It follows that   $\pi_1(Y^\Gamma_\pitchfork) \to \pi_1(Z)$ is a split injection. Therefore, without loss of generality, it suffices to prove the theorem under the additional assumption that $\iota \colon \pi_1(Y_\pitchfork) \to \pi_1(X)$ is a split injection.

		From now on, let us assume $\iota \colon \Gamma = \pi_1(Y_\pitchfork) \to \pi_1(X)$ is a split injection with a splitting morphism $\varpi\colon \pi_1(X) \to \pi_1(Y_\pitchfork) = \Gamma$.  Let $\widetilde X$ be the Galois $\Gamma$-covering space determined by  $\varpi\colon \pi_1(X) \to  \Gamma$. In particular, the restriction of the covering map $\widetilde X\to X$ on $Y_\pitchfork$ gives the universal covering space of $Y_\pitchfork$.

		Without loss of generality, assume $Y_{\pitchfork}$  is the transversal intersection  of $m$ orientable hypersurfaces $Y_j\subset X$, $1\leq j \leq m$ such that each $Y_j$  separates $\partial_{j-}$ from $\partial_{j+}$ and 
		\[ \dist(\partial_{j-}, Y_j) \geq \frac{\ell_j}{2} - \varepsilon \textup{ and } \dist(\partial_{j+}, Y_j) \geq \frac{\ell_j}{2} - \varepsilon  \]
		for some sufficiently small $\varepsilon >0$.  Furthermore, without loss of generality, we assume 
		\[  \ell_1 = \min_{1\leq j \leq m} \ell_j. \]
			Let us set 
		\[ L = \Big(\sum_{j=1}^{m} \frac{\ell_1^2}{\ell_j^2}\Big)^{1/2}.\]
		Assume  to the contrary that
		\[ \sum_{j=1}^m \frac{1}{\ell_j^2} < \frac{n^2}{
			(\frac{8}{\sqrt{3}}C + 4\pi)^2}. \]
	that is, 
		\[ \frac{1}{\ell_1^2} \cdot L^2  < \frac{n^2}{	(\frac{8}{\sqrt{3}}C + 4\pi)^2}. \]
		Therefore, we have 
		\begin{equation}\label{eq:minwidth}
			\min_{1\leq j \leq m} \ell_j = \ell_1 > \frac{L(\frac{8}{\sqrt{3}}C + 4\pi)}{n}.
		\end{equation}
	For each $1\leq j \leq m$, let $\varphi_j\colon X\to \mathbb R$ be a real-valued smooth function such that (cf. \cite[proposition 2.1]{MR532376})
		\begin{enumerate}[$(1)$]
			\item $\|d\varphi_j\| \leq  1$,
			\item $\varphi_j(x) \equiv 0$ for all $x$ between $Y_j$ and $\partial_{j-}$ with  $\dist(x, Y) \geq \frac{2\pi L}{n} + \varepsilon $, 
			\item and  $\varphi(x) \equiv \frac{4\pi L}{n}$ for all $x$ between $Y_j$ and $\partial_{j+}$ with  $\dist(x, Y) \geq \frac{2\pi L}{n} + \varepsilon$.
		\end{enumerate}

		Let us fix a sufficiently small $\varepsilon > 0 $ and denote the lift of $\varphi_{j}$ from $X$ to $\widetilde X$ still by $\varphi_j$.  Define the function
		\[   u_j(x) = \exp\big(\frac{n\ell_1}{2L\ell_j}i\varphi_j(x)\big) \]
		on $\widetilde X$. We have  
		\[ \|[\widetilde D, u_j]\| = \|du_j\|  = \frac{n\ell_1}{2L\ell_j}\|u_j\cdot d\varphi_j\| \leq \frac{n\ell_1}{2L\ell_j}  \]
		and 
		\[  \|[\widetilde D, u_j^{-1}]\|\leq \frac{n\ell_1}{2L\ell_j}.  \]
		Let $\torus^m = \sphere^1\times \cdots \times \sphere^1$ be the $m$-dimensional torus. Consider the following differential operator on $\torus^m\times \interior{\widetilde X}$:
		\[ \slashed{D} = \sum_{j=1}^{m} c_j\frac{\partial}{\partial t_j} +  \widetilde D_{t_1, t_2, \cdots, t_m}\]
		where $c_j$ is the Clifford multiplication of the unit vector $\frac{\partial}{\partial t_j}$ and $\widetilde D_{t_1, t_2, \cdots, t_m}$ is inductively defined as follows. We define
		\[ \widetilde D_{t_1} = t_1 \widetilde D + (1-t_1) u_1\widetilde Du_1^{-1}   \]
		and 
		\[ \widetilde D_{t_1, t_2, \cdots, t_k} \coloneqq t_k (\widetilde D_{t_1, \cdots, t_{k-1}})+ + (1-t_k) u_k (\widetilde  D_{t_1, \cdots, t_{k-1}} )u_k^{-1} \] 
		for $(t_1, \cdots, t_m)\in [0, 1]^m$. Here we have chosen the parametrization $\sphere^1 = [0, 1]/\{0, 1\}$.

		By the assumption $\Sc(X) \geq  n(n-1)$ and the estimates in Example \ref{ex:dirac}, we have 
		\[  \widetilde D^2 \geq \frac{n\cdot \min_{x\in X}\Sc(\widetilde  X)}{4(n-1)} \geq \frac{n^2}{4}.  \]
		By the calculation in the proof of Theorem $\ref{thm:band2}$, we have 
		\[ \widetilde D_{t_1}^2 = t_1 \widetilde D^2 + (1-t_1)u_1\widetilde D^2u_1^{-1} + t_1(1-t_1) [\widetilde D, u_1^{-1}][\widetilde D,u_1].  \]
		It follows that 
		\[  \widetilde D_{t_1}^2 \geq \frac{n^2}{4} - \frac{n^2\ell_1^2}{16L^2\ell_1^2}    \]
		Note that 
		\[ [\widetilde D_{t_1}, u_2] = t_1[\widetilde D, u_2] + (1-t_1) u_1 [\widetilde D, u_2] u_1^{-1}, \]
		which implies that 
		\[  \|[\widetilde D_{t_1}, u_2]\| \leq  \|[\widetilde D, u_2]\| \leq \frac{n\ell_1}{2L\ell_2}.  \]
		By induction, we conclude that 	
		\[  \widetilde D_{t_1, \cdots, t_k}^2 \geq \frac{n^2}{4} - \Big(\sum_{j=1}^k\frac{n^2\ell_1^2}{16L^2\ell_j^2}  \Big)   \]
		for each $ 1\leq k \leq m$. In particular, 
		\[  \widetilde D_{t_1, \cdots, t_m}^2 \geq \frac{n^2}{4} - \Big(\sum_{j=1}^m\frac{n^2\ell_1^2}{16L^2\ell_j^2}  \Big) = 
		\frac{3n^2}{16}.   \] By applying the same rescaling argument as in line $\eqref{eq:rescale}$,  we conclude that there exists $K>0$ such that 
		\begin{align*}
			\slashed{D}_\lambda^2  & =   -\sum_{j=1}^{m} \frac{\partial^2}{\partial t_j^2} + \lambda^2 \widetilde D_{t_1, \cdots, t_m}^2  + \lambda \sum_{j=1}^{m} 
			c_j \frac{\partial \widetilde D_{t_1, \cdots, t_m}}{\partial t_j} \geq  \frac{3n^2\lambda^2 }{16} - \lambda K n
		\end{align*}
on  $ C_c^\infty(\interior{\widetilde X}, \widetilde {\spinb})$. 
		
		Similarly, for each $1\leq j \leq m$, we define the operator 
		\[ \slashed{D}_j = \sum_{i=1}^{m} c_i\frac{\partial}{\partial t_i} +  \widetilde D_{t_1, \cdots, \widehat{t}_j, \cdots, t_m}\]
		where  $\widetilde  D_{t_1, \cdots, \widehat{t}_j, \cdots, t_m}$ is defined the same way as $ \widetilde D_{t_1, \cdots, t_j, \cdots, t_m}$ except that $u_j$ is replaced by the trivial unitary $v \equiv 1$.  More generally, for each subset $\Lambda \subseteq \{1, 2, \cdots, m\}$, we define the operator 
		\[ \slashed{D}_\Lambda= \sum_{i=1}^{m} c_i\frac{\partial}{\partial t_i} +  \widetilde D_{\Lambda}\]
		where  $\widetilde D_{\Lambda}$ is defined the same way as $ \widetilde  D_{t_1, \cdots, t_j, \cdots, t_m}$ except that ${u}_k$ is replaced by the trivial unitary $v \equiv 1$ for every $k\in \Lambda$. The same argument above shows that 
			\begin{align*}
			\slashed{D}_{\Lambda, \lambda}^2  &  \geq  \frac{3n^2\lambda^2 }{16} - \lambda K n
		\end{align*}
		for all $\Lambda$ and $\lambda$.

		Now we consider the doubling $\double{X} = X\cup(-X)$ of $X$ and fix a Riemannian metric on $\double{X}$ that extends the metric of $X$. Of course, this metric on $\double{X}$ generally does \emph{not} satisfy $\Sc(\double{X}) \geq n(n-1)$. Let $\widetilde{\double{X}}$ be the corresponding Galois covering of $\double X$. 
		
		We extend each unitary $u_j$ to become a unitary $\double{u}_j$  on $\widetilde{\double{X}}$ as follows. Recall that 
			\[   u_j(x) = \exp\big(\frac{n\ell_1}{2L\ell_j}i\varphi_j(x)\big) \textup{ on } X. \]
		Let $\double{X}_j$ be the ``partial'' doubling of $X$ obtained by identifying the corresponding faces $\partial_{k\pm}$ of  $X$ and $-X$ for all $1\leq k \leq m$ except the faces $\partial_{j\pm}$.   The space $\double{X}_j$ is a manifold with corners, whose boundary consists of $\partial_+(\double{X}_j)$ and $\partial_-(\double{X}_j)$. Extend the function $\varphi_j$ on the chosen copy of $X$ to a real-valued smooth function $\check{\varphi}_j$ on $\double{X}_j$ such that $\check \varphi_j(x) = 0$ in an open neighborhood of $\partial_{-}(\double{X}_j)$ in $\double X$ and $\check\varphi_j(x) = \frac{4\pi L}{n}$ in an open neighborhood of $\partial_{+}(\double{X}_j)$.
		We define  the unitary 
		\[  \check{u}_j(x) =\exp\big(\frac{n\ell_1}{2L\ell_j}i
	\check\varphi_j(x)\big)\textup{ on } \double{X}_j. \] 
		By construction, the unitary $\check{u}_j = 1$ near the boundary of $\double{X}_j$, hence actually defines a unitary\footnote{We  \emph{do not} require $\|d\check{\varphi}_j\|\leq 1$ on $\double{X}\backslash X$, where the norm $\|d\check{\varphi}_j\|$ is taken with respect to the Riemannian metric on $\double{X}$.     } on $\double{X}$, which will still be denoted by $\check{u}_j$.  
		Let us denote the lift of $\check{u}_j$ to $\widetilde{\double{X}}$ by $\double{u}_j(x)$.   
		Then $ \double{u}_j$ is a unitary on $\widetilde{\double{X}}$ whose restriction on $\widetilde X$ is $u_j$. 
		
		We consider  the following differential operator on $\torus^m\times \widetilde{\double{X}}$: 
		\[ \slashed{D}^{\double{X}} = \sum_{j=1}^{m} c_j\frac{\partial}{\partial t_j} +  \widetilde D^{\double{X}}_{t_1, t_2, \cdots, t_m}\]
		where $c_j$ is the Clifford multiplication of the unit vector $\frac{\partial}{\partial t_j}$ and $\widetilde D^{\double{X}}_{t_1, t_2, \cdots, t_m}$ is inductively defined as follows:
		\[ \widetilde D^{\double{X}}_{t_1} = t_1\widetilde  D^{\double{X}} + (1-t_1) \double{u}_1\widetilde  D^{\double{X}}\double{u}_1^{-1}   \]
		and 
		\[ \widetilde D^{\double{X}}_{t_1, t_2, \cdots, t_k} \coloneqq t_k (\widetilde D^{\double{X}}_{t_1, \cdots, t_{k-1}})+ + (1-t_k) \double{u}_k (\widetilde D^{\double{X}}_{t_1, \cdots, t_{k-1}} )\double{u}_k^{-1} \] 
		for $(t_1, \cdots, t_m)\in [0, 1]^m$. 
		More generally, for each subset $\Lambda \subseteq \{1, 2, \cdots, m\}$, we define the operator 
		\[ \slashed{D}^{\double{X}}_\Lambda= \sum_{i=1}^{m} c_i\frac{\partial}{\partial t_i} +  \widetilde D^{\double{X}}_{\Lambda}\]
		where  $\widetilde D^{\double{X}}_{\Lambda}$ is defined the same way as $ \widetilde D^{\double{X}}_{t_1, \cdots, t_j, \cdots, t_m}$ except that $\double{u}_k$ is replaced by the trivial unitary $\double{v} \equiv 1$ for every $k\in \Lambda$. See Figure \ref{fig:rel} for the case where $m =2$.

		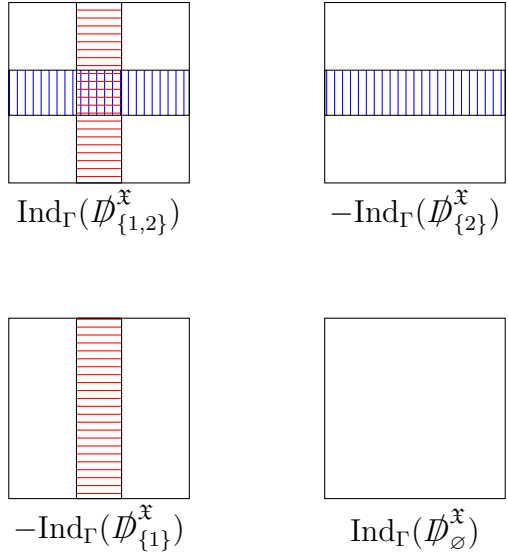
\begin{figure}
			\centering 
			\begin{tikzpicture}[scale = 0.6]
				\coordinate (A) at (0, 0);
				\coordinate (B) at (4, 0);
				\coordinate (C) at (4, 4);
				\coordinate (D) at (0, 4);	
				\draw (A) rectangle (C);  
				\filldraw[pattern=vertical lines, pattern color=blue] ($(A) + (0, 1.5)$)  rectangle ($(A) + (4, 2.5)$);
				\filldraw[pattern=horizontal lines, pattern color=red] ($(A) + (1.5, 0)$)  rectangle ($(A) + (2.5, 4)$);
				\filldraw[black] ($(A) + (2,-1.4)$) circle (0pt) node[anchor=south] {$  \ind_{\Gamma}(\slashed{D}^{\double X}_{\{1, 2\}})$};

				\draw ($(A) + (7, 0)$) rectangle ($(C) + (7,0)$); 
				\filldraw[pattern=vertical lines, pattern color=blue] ($(A) + (0, 1.5) + (7, 0)$)  rectangle ($(A) + (4, 2.5) + (7, 0)$);
				\filldraw[black] ($(A) + (2, -1.4) + (7, 0)$) circle (0pt) node[anchor=south] { $-\ind_{\Gamma}(\slashed{D}^{\double X}_{\{ 2\}})$};

				\draw ($(A) + (0, -7)$)  rectangle ($(C) + (0, -7)$);  
				\filldraw[pattern=horizontal lines, pattern color=red] ($(A) + (1.5, 0) + (0, -7)$)  rectangle ($(A) + (2.5, 4) + (0, -7)$);
				\filldraw[black] ($(A) + (2,-1.4) + (0, -7)$) circle (0pt) node[anchor=south] {$ - \ind_{\Gamma}(\slashed{D}^{\double X}_{\{1\}})$};
				
				\draw ($(A) + (7, -7)$)  rectangle ($(C) + (7, -7)$);  
				\filldraw[black] ($(A) + (2,-1.4) + (7, -7)$) circle (0pt) node[anchor=south] {$  \ind_{\Gamma}(\slashed{D}^{\double X}_{\varnothing})$};
				
			\end{tikzpicture}
			\caption{An illustration of the indices in the $m=2$ case where the horizontal (red) lines represent the unitary $\double u_1$ and the vertical (blue) lines represent the unitary $\double u_2$ }
			\label{fig:rel}
		\end{figure}

		Let us compute the following index 
			\begin{equation}\label{eq:indexeq}
				\sum_{\Lambda \subseteq \{1, 2, \cdots, m\}  } (-1)^{|\Lambda|} \cdot  \ind_{\Gamma}(\slashed{D}^{\double{X}}_\Lambda) 
					\end{equation}
			in $ KO_{n-m}(C^\ast_{\max}(\Gamma))$, where $|\Lambda|$ is the cardinality of the set $\Lambda$. 	
				Since $\double X$ is a closed manifold, the index in line  $\eqref{eq:indexeq}$ does not change if we deform the unitaries $\double{u_j}$ through a continuous family of unitaries. In particular, we can deform the unitaries $\double{u}_j$ through a continuous family of unitaries so that each $\double{u}_j$  becomes trivial (that is, equal to $1$) outside a small neighborhood of the hypersurface $\double{Y}_j$ in $\double{X}$, where $\double{Y}_j$ is the doubling of $Y_j$.   Now we identify a small tubular neighborhood  of $Y_{\pitchfork}$ in $\double X$ with an open set in $Y_{\pitchfork}\times\torus^{m}$.   By the usual relative higher  index theorem for closed manifolds (cf. \cite{UB95}\cite{MR3122162}) or alternatively the proof of Theorem $\ref{thm:relative}$, we can reduce the computation to  the corresponding operators on  the closed manifold  $Y_{\pitchfork}\times \torus^{m}$.  Hence it remains to compute the index 
		\[ \sum_{\Lambda \subseteq \{1, 2, \cdots, m\}  } (-1)^{|\Lambda|} \cdot  \ind_{\Gamma}(\slashed{D}^{Y_{\pitchfork}\times\torus^{m}}_\Lambda) \]
		where $\slashed{D}^{Y_{\pitchfork}\times\torus^{m}}_\Lambda$ is the obvious analogue of $\slashed{D}^{\double{X}}_\Lambda$. Now to simplify the computation even further, we  deform the metric on $Y_{\pitchfork}\times \torus^{m}$ to a product metric. In this case, the operator  $\slashed{D}^{Y_{\pitchfork}\times\torus^{m}}$ becomes 
		\[  \sum_{j=1}^{m}\big(c_j \frac{\partial}{\partial t_j} + \double{u}_{j}D^{\sphere^1}\double{u}_{j}^{-1}\big)\hotimes 1 + 1\hotimes D^{Y_\pitchfork} \]
		on the space $\torus^m\times Y_{\pitchfork} \times \torus^m$, where without loss of generality we can assume $\double{u}_j$ to be the smooth function obtained by projecting to the $j$-component of $\torus^m$: 
		\[ Y_{\pitchfork}\times \torus^{m} \to \sphere^1 \subset \mathbb C. \]
		The operator $\sum_{j=1}^{m}\big(c_j \frac{\partial}{\partial t_j} + \double{u}_{j}D^{\sphere^1}\double{u}_{j}^{-1}\big)$ has index $1$ (cf. \cite[Section 7]{A-P-S76}). Therefore, it follows that 
		\[  \ind_{\Gamma}(\slashed{D}^{Y_{\pitchfork}\times \torus^{m}}) = \ind_{\Gamma}(D^{Y_\pitchfork}) \in K_{n-m}(C_{\max}^\ast(\Gamma)). \]
		Similarly, one can show that 
		\[ \ind_{\Gamma}(\slashed{D}^{Y_{\pitchfork}\times \torus^{m}}_\Lambda) = 0  \]
		whenever $\Lambda $ is a proper subset of $\{1, 2, \cdots, m\}$. To summarize, we have 
		\[ \sum_{\Lambda \subseteq \{1, 2, \cdots, m\}  } (-1)^{|\Lambda|} \cdot  \ind_{\Gamma}(\slashed{D}^{\double{X}}_\Lambda)  =  \ind_{\Gamma}(D^{Y_\pitchfork}).  \]

		On the other hand, we have (cf. line \eqref{eq:minwidth})
		\[ \min_{1\leq j \leq m} \ell_j = \ell_1 > \frac{L(\frac{8}{\sqrt{3}}C + 4\pi)}{n}.\]
		Furthermore, by appropriately choosing the metric on $\double X$ that extends the metric on $X$,  we can assume that 
		\[   \supp_{\double X\backslash X}(\double u_j - 1) \textup{ and } \supp_{\double X\backslash X}(\double u_k - 1) \textup{ are at least  $\frac{L(\frac{4}{\sqrt{3}}C + 4\pi)}{n}$ apart}\]
		for all $j\neq k$, where $\supp_{\double X\backslash X}(\double u_j - 1)$ is the support of $(\double u_j - 1)$ in $\double X\backslash X$. Now we apply  the same argument as in the proof of Theorem \ref{thm:relative} (and  Remark \ref{rm:propspeed}) and iterate the difference construction from Lemma \ref{lm:diff} below.  It follows that 
		\[ \sum_{\Lambda \subseteq \{1, 2, \cdots, m\}  } (-1)^{|\Lambda|} \cdot  \ind_{\Gamma}(\slashed{D}^{\double{X}}_\Lambda) =  0  \]
		in $K_{n-m}(C_{\max}^\ast(\Gamma))$.   We arrive at a contradiction, since $\ind_{\Gamma}(D^{Y_\pitchfork})\neq 0$ by assumption. This finishes the proof. 
		
	\end{proof}

\begin{lemma}[cf. {\cite[section 6]{GKGY06}}]\label{lm:diff} 
	Let $p_1$ and $p_2$ be two idempotents in a Banach algebra $B$. Then we have 
	\[ [p_1] - [p_2] = [E(p_1, p_2)] - [E_0]  \] 
	in  $K_0(B)$, where 
	\begin{align}
		E(p_1, p_2)   & = \begin{pmatrix} 1 + p_2(p_1-p_2) p_2 & 0 & p_2p_1(p_1-p_2) & 0 \\ 0 & 0 & 0 & 0 \\ (p_1-p_2)p_1p_2 & 0 & (1-p_2)(p_1-p_2)(1-p_2)  & 0 \\ 0 & 0 & 0 & 0 \end{pmatrix} \label{eq:diff}
	\end{align}
	and 
	\[  E_0 = \begin{psmallmatrix} 1 & 0 & 0 & 0 \\ 0 & 0 & 0 & 0 \\ 0 & 0 & 0  & 0 \\ 0 & 0 & 0 & 0 \end{psmallmatrix}.   \]
\end{lemma}
   \begin{proof}
   		Consider the invertible element
   	\[ U = \begin{pmatrix} p_2 & 0 & 1-p_2 & 0 \\ 1-p_2 & 0 & 0 & p_2 \\ 0 & 0 & p_2  & 1-p_2 \\ 0 & 1 & 0 & 0 \end{pmatrix} \]
   	whose inverse is given by
   	\[ U^{-1} = \begin{pmatrix} p_2 & 1-p_2 & 0 & 0 \\ 0 & 0 & 0 & 1 \\ 1-p_2 & 0 & p_2  & 0 \\ 0 & p_2 & 1-p_2 & 0 \end{pmatrix}. \]
   	A direct computation shows that
   	\[ E(p_1, p_2) = U^{-1} \begin{pmatrix} p_1 & 0 & 0 & 0 \\ 0 & 1-p_2 & 0 & 0 \\ 0 & 0 & 0  & 0 \\ 0 & 0 & 0 & 0 \end{pmatrix} U. \]
   	This proves the lemma. 
   \end{proof}

	\section{Proofs of  Theorems \ref{thm:sphere-ad} and  \ref{thm:hemisphere}} \label{sec:rigidity}
	In this section, we prove Theorems \ref{thm:sphere-ad} and \ref{thm:hemisphere}. Let us first recall the following notion of subsets with the wrapping property, which was introduced in Definition \ref{def:wrap-intro}.

	\begin{definition}[Subsets with the wrapping property, cf. Definition \ref{def:wrap-intro}]\label{def:wrap}
		A subset $\Sigma$ of the standard unit sphere $\sphere^n$ is said to have \emph{the wrapping property} if  $\Sigma$ is strongly non-separating (cf. Definition \ref{def:nonsep}) and furthermore there exists a  smooth distance-contracting map $\Phi\colon \sphere^n \to \sphere^n$ such that the following are satisfied:
		\begin{enumerate}[(1)]	  
			\item[(1a)] if $n$ is even,  $\Phi$ equals the identity map on  $N_{\varepsilon}(\Sigma)$; 
			\item[(1b)] if $n$ is odd, $\Phi$ equals either the identity map or the antipodal map on each of the connected components of  $N_{\varepsilon}(\Sigma)$;   
			\item[(2)] and\footnote{For example, if $\Phi$ is not surjective, then clearly $\deg(\Phi) = 0\neq 1$.} $\deg(\Phi) \neq 1$.
		\end{enumerate}
\end{definition} 

Loosely speaking, the class of subsets in $\sphere^n$ with the wrapping property includes  all ``reasonable" geometric subsets of $\sphere^n$ whose sizes are ``relatively small". For example, Lemma \ref{lm:wrap} below gives a sufficient geometric condition for a subset to satisfy the wrapping property.  Let us first fix some terminology. 

\begin{definition}
	Consider the canonical embedding of the unit sphere $\sphere^n$ inside the Euclidean space $\mathbb R^{n+1}$. For each unit vector $v\in \mathbb R^{n+1}$, denote by $\mathbb V^{\perp}_v$  the linear subspace of $\mathbb R^{n+1}$ that is orthogonal to $v$. We define an equator $\equator$  of $\sphere^n$ to be the intersection of  $\mathbb V^{\perp}_v$ and $\sphere^n$ for some unit vector $v\in \mathbb R^{n+1}$.
\end{definition}  

\begin{lemma}\label{lm:wrap}
	Let  $\Sigma$ be  a strongly non-separating subset of $\sphere^n$.  If  $N_\varepsilon(\Sigma)$  is contained in a geodesic ball of radius $ < \frac{\pi}{2}$  for some \textup{(}hence for all\textup{)}  sufficiently small $\varepsilon>0$,  then $\Sigma$  has the wrapping property. 
\end{lemma}
\begin{proof}
	
	By assumption, for each sufficiently small $\varepsilon >0$, there exists a geodesic ball $B$ of radius $r <\frac{\pi}{2}$ that contains $N_\varepsilon(\Sigma)$. Without loss of generality, we assume that there is an equator $\equator$ such that $B$ is contained in a hemisphere determined by $\equator$ and $\dist(B, \equator) > 2\varepsilon$.     Let us denote the center of $B$ by $x_0$.   Consider all geodesics in $\sphere^n$ of length $\leq \pi $  that originate from $x_0$, that is, all the shortest geodesics starting at $x_0$ and ending at the antipodal point of $x_0$.  Now we shall ``wrap" the geodesics to define a distance-contracting map $\Phi\colon \sphere^n \to \sphere^n$ such that $\Phi$ equals the identity map on $B$ and its image $\Phi(\sphere^n)$ lies in the hemisphere that contains $B$.  In particular, $\Phi$ is not surjective, hence $\deg(\Phi) =0$. 
	
	More precisely, let us first consider a smooth function $f'\colon [-\frac{\pi}{2}, \frac{\pi}{2}] \to [-1, 1]$ such that (cf. Figure $\ref{fig:graph}$)
	\begin{enumerate}[(i)]
		\item $f'$ is odd, that is, $f'(-t) = -f'(t)$, 
		\item $f'(t) = -1$ for all $t\in [\varepsilon, \frac{\pi}{2}]$,
		\item and $f'(t) \leq 0$ for all $t\in [0, \frac{\pi}{2}]$. 
	\end{enumerate}
	\begin{figure}
		\centering 
		\begin{tikzpicture}
			\begin{axis}[xmin=-3, xmax=3, ymin=-1.5, ymax=1.5,
				xtick={  -0.6, 0.6}, ytick={-1, 1}, xticklabels={  $-\varepsilon$, $\varepsilon$},
				yticklabels={},
				scale=1, restrict y to domain=-5:5,
				axis x line=center, axis y line= center,
				samples=40]

				%
				%
				%
				%

				\addplot[black, samples=100, smooth, domain=-1.2:1.2, thick]
				plot (\x, {(tanh(5*(-\x)))});

			\end{axis}
		\end{tikzpicture}

		\caption{The graph of $f'$}
		\label{fig:graph}
	\end{figure}
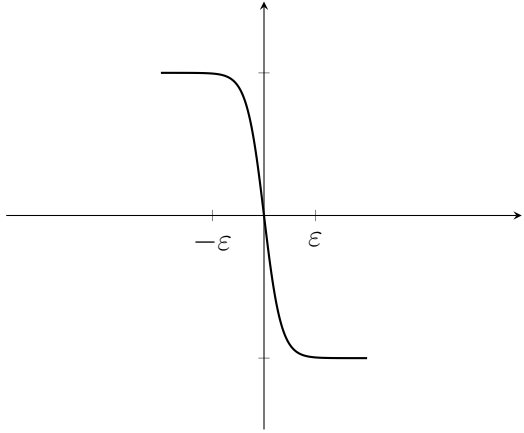
	Define $f\colon [-\frac{\pi}{2}, \frac{\pi}{2}] \to \mathbb R$ by setting
	\[  f(s) = -\frac{\pi}{2} + \int_{-\frac{\pi}{2}}^s f'(t) \dif t.  \]
	For each shortest geodesic $\gamma$ going from $x_0$ to its antipodal point, we parametrize $\gamma$ by its arc length so that the intersection point of $\gamma$ with the equator $ \equator$  becomes the origin of the interval $ [-\frac{\pi}{2}, \frac{\pi}{2}]$ and $x_0$  becomes $-\frac{\pi}{2}$ with respect to the parametrization. Now we define $\Phi_{\equator} \colon \sphere^n \to \sphere^n$ by setting
	\begin{equation}\label{eq:wrapmap}
		\Phi_{\mathbb E}(\gamma(t)) = \gamma(f(t))  
	\end{equation}
	for each $t\in [-\frac{\pi}{2}, \frac{\pi}{2}]$. For later references, let us call  $\Phi_{\mathbb E}$ a wrapping map along the equator $\mathbb E$. For brevity, let us denote it by $\Phi$. By construction, the wrapping map 
	\[ \Phi \colon \sphere^n \to \sphere^n \]
	is a smooth\footnote{Due to the specific properties of $f$, the map $\Phi$ is smooth everywhere. In particular, $\Phi$ is smooth at the antipodal point of $x_0$.} distance-contracting map such that  $\Phi$ equals the identity map on $B$ and  $\Phi(\sphere^n)$ lies in a hemisphere, hence $\deg(\Phi) = 0\neq 1$.  This finishes the proof.

\end{proof}

\begin{example}
	By Lemma \ref{lm:wrap}, the following subsets of  $\sphere^n$ have the wrapping property:
	\begin{enumerate}[(a)]
		\item an  open or closed geodesic ball of radius $<\frac{\pi}{2}$,
		\item any  compact simplicial complex of codimension $\geq 2$  that is contained in a geodesic ball of radius $<\frac{\pi}{2}$. 
	\end{enumerate}
\end{example}

For odd dimensional spheres, the following collection of subsets also satisfy the wrapping property. 
\begin{lemma}\label{lm:wrap-antipodal}
	Let  $\Sigma$ be  a strongly non-separating subset of $\sphere^{2k+1}$. Let $\{\equator_j\}_{1\leq j \leq 2k+2}$ be a collection of $(2k+2)$ mutually orthogonal  equators of $\sphere^{2k+1}$ so that they divides $\sphere^{2k+1}$ into $2^{(2k+2)}$ regions.    If $N_\varepsilon(\Sigma)$  is contained in an antipodal pair of such regions for some \textup{(}hence for all\textup{)}  sufficiently small $\varepsilon>0$,  then $\Sigma$  satisfied the wrapping property. 
\end{lemma}
\begin{proof}
	Since by assumption $N_\varepsilon(\Sigma)$  is contained in a pair of antipodal regions determined the equators $\{\equator_j\}_{1\leq j \leq 2k+2}$, we can choose a wrapping map $\Phi_{\equator_j}$ associated to $\equator_j$ as defined in line \eqref{eq:wrapmap} for each $1\leq j \leq {2k+2}$ such that their composition 
	\[ \Phi \coloneqq \Phi_{\equator_1}\circ \Phi_{\equator_2} \circ \cdots  \circ \Phi_{\equator_{2k+2}}\]
	satisfies the desired properties (1b) and (2) in Definition \ref{def:wrap}. This finishes the proof. 
\end{proof}

\begin{example}
	By Lemma \ref{lm:wrap-antipodal}, the following subsets of an odd dimensional sphere $\sphere^{2k+1}$ have the wrapping property:
	\begin{enumerate}[(a)]
		\item  a pair of antipodal points on $\sphere^{2k+1}$,
		\item  a pair of antipodal geodesic balls of radius $<\frac{\pi}{6}$ in $\sphere^{2k+1}$.
		\item any  compact simplicial complex of codimension $\geq 2$  that is contained in a pair of antipodal geodesic balls of radius $<\frac{\pi}{6}$  in $\sphere^{2k+1}$.
	\end{enumerate}
\end{example}

For a given $\varepsilon>0$, the geometry of the $\varepsilon$-neighborhood $N_\varepsilon(\Sigma)$ of $\Sigma$ can be very wild, in particular at the boundary $\partial \overbar{N_\varepsilon}$. However,  by enlarging or shrinking $N_\varepsilon(\Sigma)$ if necessary,  we can in fact always find small neighborhoods of $\Sigma$ that are manifolds with boundary or manifolds with corners. 

\begin{lemma}\label{lm:mfldc}
	Let $\Sigma$ be a subset of $\sphere^n$. Then for any sufficiently small $\varepsilon >0$, there is a  subspace $X_\varepsilon\subset \sphere^n$ with  $\sphere^n\backslash N_{2\varepsilon}(\Sigma) \subset X_\varepsilon \subset \sphere^n\backslash N_\varepsilon(\Sigma)$ such that $X_\varepsilon$ is an $n$-dimensional compact manifold with corners. Furthermore, if $N_\varepsilon(\Sigma)$ is non\nobreakdash-separating for all sufficiently small $\varepsilon>0$, then $X_\varepsilon$ can also be chosen to be path-connected for all sufficiently small $\varepsilon>0$. 
\end{lemma}	
\begin{proof}
	Let $\mathcal U = \{U_j\}$ be an open cover of $\overbar{N_{\varepsilon}(\Sigma)}$ consisting of geodesically convex balls of radius $\leq \frac{\varepsilon}{2}$. Note that $\overbar {N_{\varepsilon}(\Sigma)}$ is closed in $\sphere^n$, hence compact. It follows that $\overbar {N_{\varepsilon}(\Sigma)}$ admits a finite open cover $\mathcal V$ consisting of finitely many members of $\mathcal U$. Without loss of generality, we assume 
	\[V\cap \overbar {N_{\varepsilon}(\Sigma)} \neq \varnothing  \] 
	for each member $V$ of $\mathcal V$. Denote by $W$  the union of all members of  $\mathcal V$. Then the closure $\overbar{W}$ of $W$ is contained in $N_{2\varepsilon}(\Sigma)$.  
	
	Define $X_\varepsilon$ to be $\sphere^n\backslash W$. By construction,\footnote{We do not rule out the possibility that $X_\varepsilon$ could be the empty set.} $X_\varepsilon$ is  an $n$-dimensional compact manifold with corners under the metric inherited from $\sphere^n$ such that 
	\[ \sphere^n\backslash N_{2\varepsilon}(\Sigma) \subset X_\varepsilon \subset \sphere^n\backslash N_\varepsilon(\Sigma). \] 
	 Furthermore, the above construction shows that  if $\sphere^n\backslash N_{2\varepsilon}(\Sigma)$ is path-connected, then $X_\varepsilon$ is path-connected. 
\end{proof}

Now we are ready to prove Theorem  \ref{thm:sphere-ad}, which answers positively an open question of Gromov, cf. \cite[page 687, specific problem]{MR3816521} and \cite[Section 3.9]{Gromov:2019aa}.

\begin{theorem}[Theorem \ref{thm:sphere-ad}]\label{thm:sphere-ad2}
	Let $\Sigma$  be a  subset of the standard unit sphere $\sphere^n$ contained in a geodesic ball of radius $r < \frac{\pi}{2}$.  Let $(X, g_0)$ be the standard unit sphere $\sphere^n$ minus $\Sigma$. If a \textup{(}possibly incomplete\textup{)} Riemannian metric $g$ on $X$ satisfies that  
\begin{enumerate}[$(1)$]
	\item there is a $\lambda_n$-Lipschitz  homeomorphism $\varphi\colon (X, g) \to (X, g_0)$,  	 
	\item and $\Sc(g) \geq n(n-1) = \Sc(g_0)$, 
\end{enumerate}
then 
\[ \lambda_n \geq \sqrt{1 - \frac{C_r}{n^2}} \]
where\footnote{If $n=\dim \sphere^n$ is odd, our proof of Theorem \ref{thm:sphere-ad2} in fact shows that we can improve $C_r$ to be  
	\[ \frac{4C^2}{(\frac{\pi}{2} - r)^2} \textup{ instead of } \frac{8C^2}{(\frac{\pi}{2} - r)^2}. \]  } 
\[ C_r = \frac{8C^2}{(\frac{\pi}{2} - r)^2}  \]
and $C$ is a universal constant from Theorem $\ref{thm:relative-intro}$. Consequently, the lower bound for $\lambda_n$ approaches $1$, as $n\to \infty$. 
\end{theorem}
\begin{proof}
	We prove the theorem by contradiction. Assume to the contrary that 
	\[ \lambda_n  < \sqrt{1 - \frac{C_r}{n^2}}. \]	To avoid ambiguity, let us denote $(X, g_0)$ by $\underline{X}$ for the rest of the proof.   
	
	Let us first prove the even dimensional case. Recall the $\ccl_{n}$-Dirac bundle  $E_0$ over $\mathbb S^n$: 
	\begin{equation}\label{eq:cliffbdle}
		E_0 = P_{\spin}(\mathbb S^n) \times_{\ell} \ccl_{n}
	\end{equation} 
	where $\ell\colon \spin_n \to \End(\ccl_n)$ is the representation given by left multiplication. Equip $E_0$ with the canonical Riemannian connection  determined by the presentation $\ell \colon P_\spin(\mathbb S^n) \to \End(\cl_n)$. Furthermore, when $n$ is even, $E_0$  carries a natural $\mathbb Z/2$-grading $E_0 = E_0^+ \oplus E_0^-$. By the Atiyah-Singer index theorem  \cite{MAIS68b}, the index of the twisted Dirac operator $D^{\sphere^n}_{E_0^+}$ is nonzero (cf. \cite[equation (4.7)]{MR1600027}). 
	
	We shall give an explicit description of the bundle $E_0$ as a sub-bundle of a trivial vector bundle over $\sphere^n$ so that $E_0$ can be viewed a projection $p$ in $M_k(C(\sphere^n)) = M_k(\mathbb C) \otimes C(\sphere^n)$, where $C(\sphere^n)$ is the $C^\ast$-algebra of continuous functions on $\sphere^n$. Consider the canonical embedding of the unit sphere $\sphere^n$ inside the Euclidean space $\mathbb R^{n+1}$. Let $V = \mathbb R^{n+1}\times \ccl_{n+1}$ be the canonical $\ccl_{n+1}$-Dirac bundle over $\mathbb R^{n+1}$. Clearly, $V$ is a trivial vector bundle. Let us still denote by $V$ the restriction of $V$ on $\sphere^n$. Then we see that $E_0$ is a sub-bundle of $V$. Denote by $v$ the outward unit normal vector field  of $\sphere^n$ in $\mathbb R^{n+1}$. Then $E_0$ is isomorphic to the sub-bundle of $V$ determined by the following Bott projection\footnote{To be precise, the $K$-theory class of $\botts$ is a nonzero multiple of  the actual Bott generator of $K^0(\sphere^{n})$. } 
\begin{equation}\label{eq:bottproj}
\botts =  \frac{i\hspace{0.07em}\overbar{c}(v) + 1}{2}
\end{equation} 
	where $\overbar{c}(v)$ is the Clifford multiplication of $v$ on $V = \sphere^n\times \ccl_{n+1}$.

By assumption, $\Sigma$  is contained in a geodesic ball $B_r(x_0)$ centered at $x_0$ of radius $r < \frac{\pi}{2}$. Let us define 
\[ \underline{X}_{\varepsilon} = \sphere^n \backslash \interior{B}_{r+\frac{\varepsilon}{2}}(x_0) \]
where $\interior{B}_{r+\frac{\varepsilon}{2}}(x_0)$ is the open  geodesic ball centered at $x_0$ of radius $ (r+\frac{\varepsilon}{2})$.   By the proof of Lemma \ref{lm:wrap},  there exists a  smooth distance-contracting map $\Phi\colon \sphere^n \to \sphere^n$ such that 
	\begin{enumerate}
		\item[(1)] $\Phi$ equals the identity map on the $(\frac{\pi}{2} -r- \varepsilon)$-neighborhood of $\sphere^n\backslash \underline{X}_\varepsilon$; 
		\item[(2)] and $\deg(\Phi) \neq 1$. 
	\end{enumerate}

In order to apply the relative index theorem (Theorem \ref{thm:relative}), we shall view $X_\varepsilon$  as a (topological) subset of the $n$-dimensional sphere. Since   $(X_\varepsilon, g)$ is an $n$-dimensional manifold with boundary, we can extend the Riemannian metric $g$ on $X_\varepsilon$ to a Riemannian metric on the sphere. Let us denote by $\double S$ the resulting $n$-dimensional sphere with this new metric $g_{\double S}$. Of course, the metric $g_{\double S}$ generally  does \emph{not} satisfy the Lipschitz bound and  scalar curvature bound on the complement of $X_\varepsilon$ in $\double S$, when compared to the standard metric $g_0$ on $\sphere^n$. Consider the (set-theoretic) identity map 
\[ \idsp \colon \double{S} \to \sphere^n.  \]
The pullback bundles of $V$ by the map ${\idsp \colon \double{S} \to \sphere^n}$ and the map $\Phi\circ \idsp \colon \double{S} \to \sphere^n$ are identical, since $V$ is a trivial vector bundle with its canonical trivial connection.  We shall denote this pullback bundle on $\double S$ by $W = \double S \times \ccl_{n+1}$ from now on. Let $\spinb$ be the spinor bundle of $(\double S, g_{\double S})$.

Let $\bott_1 = (\idsp)^\ast(\botts) $ and $\bott_2 = (\Phi\circ \idsp)^\ast(\botts)$ be the projections induced from the Bott projection $\botts$ on $\sphere^n$,  by the maps $\idsp$ and $\Phi\circ \idsp$ respectively.  By construction, we have   $\bott_1 = \bott_2$ on the \mbox{$(\frac{\pi}{2} - r -\varepsilon)$-neighborhood} of $\double S\backslash X_\varepsilon$.    The projections $\bott_1$ and $\bott_2$ can be viewed as endomorphisms of the bundle $\spinb\otimes W$. More precisely, the bundle homomorphism  $1\otimes \bott_j\colon \spinb\otimes W \to \spinb\otimes W$ satisfies that $(1\otimes \bott_j)^2 = 1\otimes \bott_j$ and $(1\otimes \bott_j)^\ast = 1\otimes \bott_j$, for $j = 1, 2$.  Now consider the twisted Dirac operators $D_{\bott_j} \coloneqq \bott_j
D \bott_j$. Furthermore, since $n$ is even, the bundle $\bott_jW$ carries a natural $\mathbb Z/2$-grading  inherited from the $\mathbb Z/2$-grading on $E_0$. We have  
\[ D_{\bott_j} = \begin{pmatrix}
	0 & D^-_{\bott_j} \\
	D^+_{\bott_j} & 0
\end{pmatrix} \]  with respect to the decomposition $\bott_jW = (\bott_jW)^+ \oplus (\bott_jW)^-$.

 The commutator $[D, \bott_j]$ is an endomorphism of the bundle $\spinb\otimes W$. Denote by $[D, \bott_j]_x \colon (\spinb\otimes W)_x \to (\spinb\otimes W)_x $  the endomorphism at the point $x\in \double S$. A key step of the proof is the following estimate for the operator norm of  $[D, \bott_j]_x$ for every point $x\in {X}_\varepsilon$.
	
	For each $x\in {X}_{\varepsilon}$, we can choose a local $g_0$-orthonormal tangent frame $\{\underline{e}_1, \cdots, \underline{e}_n\}$ for $T\underline{X}_{\varepsilon}$ and a local $g$-orthonormal tangent frame $\{{e}_1, \cdots, {e}_n\}$ for $T{X}_{\varepsilon}$ near $x$ such that for each $1\leq k \leq n$, we have 
	\[ \idsp_\ast( e_k) = \mu_k \hspace{0.07em} \underline{e}_k \]
	for some $\mu_k \geq 0$. Since  $\idsp\colon (X_{\varepsilon}, g) \to (\underline{X}_{\varepsilon}, g_0)$ is $\lambda_n$-Lipschitz, we have $\mu_k\leq \lambda_n$ for all $1\leq k \leq n$. If we write 
	\[  D = \sum_{k=1}^n c(e_k)\nabla_{e_k}, \]
	then we have
	\[ \|[D, \bott_1]_x\| = \big\|\sum_{k=1}^n \big[\mu_k c(\underline{e}_k)\nabla_{\underline{e}_k}, \botts\big]_x\big\|. \]
	A similar conclusion holds for $\bott_2$, since  $\Phi\circ \idsp\colon (X_{\varepsilon}, g) \to \sphere^n$ is also $\lambda_n$-Lipschitz.
	\begin{claim}\label{claim:combd} We have 
		\[  \|[ D, \bott_j]_x\| \leq \frac{n \cdot \lambda_n}{2} \]
		for all $x\in X_\varepsilon$ and  for both $j =1, 2$. 
	\end{claim}
	By the discussion above, we need to estimate 
	\begin{equation}\label{eq:commutator}
		\big\|\sum_{k=1}^n \big[\mu_k c(\underline{e}_k)\nabla_{\underline{e}_k}, \botts\big]_x\big\|
	\end{equation} 
	for each $x\in \sphere^n$.  Recall that $v$ is the outward unit normal vector field of $\sphere^n$ in $\mathbb R^{n+1}$. In particular, at a point $x = (x_1, x_2, \cdots, x_{n+1})\in \sphere^n\subset \mathbb R^{n+1}$,  we have 
	\[ \overbar{c}(v)_x = \sum_{k=1}^{n+1} x_{k} \overbar{c}_k \]
	where $\overbar{c}_j$ is the Clifford multiplication of the unit vector $ \frac{\partial}{\partial x_j}$ on $V= \sphere^n\times \ccl_{n+1}$ from the right. 
	Since $\SO(n+1)$ acts  transitively on $\sphere^n$, it suffices to estimate the term in line \eqref{eq:commutator} at the point $\underline{x} =(0, \cdots,0 ,  1)\in \sphere^n\subset \mathbb R^{n+1}$. At this point $\underline{x}$,  after a local coordinate change if necessary,  we have\footnote{Here the term $\partial/\partial x_{n+1}$ does not appear, since it is in the normal direction. } 
	\[  \sum_{k=1}^n \mu_k c(\underline{e}_k)\nabla_{\underline{e}_k} =  \sum_{k =1}^{n} \mu_k c_k \frac{\partial}{\partial x_k}.  \]
	where $c_k$ is the Clifford multiplication of the unit vector $ \frac{\partial}{\partial x_j}$ on the spinor bundle of $\sphere^n$ from the left. 
	We conclude that 
	\[ \sum_{k=1}^n \big[\mu_k c(\underline{e}_k)\nabla_{\underline{e}_k}, \botts\big]_{\underline{x}} =  \frac{i}{2} \sum_{k=1}^n \big[\mu_k c_k \frac{\partial}{\partial x_k}, \overbar{c}(v)\big]_{\underline{x}} = \frac{i}{2} \sum_{j=1}^{n} \mu_k c_k\otimes \overbar{c}_k.\]
	Since $\|c_k\otimes \overbar{c}_k\| = 1$ for all $1\leq k \leq n$, it follows that  
	\[ \big\|\sum_{k=1}^n \big[\mu_k c(\underline{e}_k)\nabla_{\underline{e}_k}, \botts\big]_{\underline{x}}\big\| \leq  \frac{\sum_{k=1}^{n}\mu_k}{2} \leq \frac{n\cdot \lambda_n}{2}. \]
	 This proves the claim.

	 For brevity, let us write  $\bott$ in place of $\bott_j$ in the following estimation. We have 
	\begin{align*}
		\langle \bott D\bott f,  \bott D \bott f\rangle & = \langle  \bott D \bott D \bott f,  \bott f\rangle \\
		& = \langle  \bott [D,  \bott ]D \bott f,  \bott f\rangle + \langle  \bott D^2 \bott  f,  \bott f\rangle \\
		& = -\langle D\bott f, [D, \bott]\bott f\rangle + \langle \bott D^2\bott f, \bott f\rangle \\
		& \geq -\frac{1}{2}\langle D\bott f, D\bott f\rangle - \frac{1}{2} \langle [D, \bott]\bott f, [D, \bott ]\bott f\rangle +   \langle \bott D^2\bott  f, \bott f\rangle \\
		& \geq   \frac{1}{2} \langle D^2\bott f, \bott f\rangle - \frac{1}{2}\langle [D, \bott]\bott f, [D, \bott ]\bott f\rangle   
	\end{align*}
	By the inequality in line \eqref{eq:lichbd}, we have 
	\[ \langle D^2\bott f, \bott f\rangle \geq \frac{n}{n-1} \big\langle \frac{\kappa}{4}\bott f, \bott f\big\rangle,\]
	where $\kappa \coloneqq \Sc(g)$. 
	It follows that 
	\[ (\bott D\bott)^2  \geq  \frac{1}{2} \Big( \frac{n\hspace{0.07em}\kappa}{4(n-1)} - [D, \bott]^\ast [D, \bott] \Big) \textup{ on } C^\infty_c(\interior{X}_\varepsilon,  \spinb\otimes \bott W).    \]
	Here $C^\infty_c(\interior{X}_\varepsilon,  \spinb\otimes \bott W)$ is the space of compactly supported smooth sections of the sub-bundle $\spinb\otimes \bott W \subset \spinb\otimes W$. 
	By assumption, we have $\kappa = \Sc(g) \geq n(n-1)$.  It follows from Claim \ref{claim:combd} that  
	\begin{equation}\label{eq:lowerbd-sphere}
		\|{D}_{\bott_j} v\| \geq n\sqrt{\frac{(1-\lambda_n^2)}{8}} \|v\|
	\end{equation} 
	for all $v\in C_c^\infty(\interior{X}_\varepsilon, \spinb\otimes \bott_j W)$ and for both $j =1, 2$.  Furthermore,  the same conclusion from line \eqref{eq:lowerbd-sphere} also holds for both  $D^{+}_{\bott_j}$ and $ D^{-}_{\bott_j}$ on $C_c^\infty(\interior{X}_\varepsilon, \spinb\otimes \bott_j W)$.

Since we have assumed that  
\[ \lambda_n  < \sqrt{1 - \frac{C_r}{n^2}} \textup{ with } C_r = \frac{8C^2}{(\frac{\pi}{2} - r)^2} \]
it follows that 
\[ \big(\frac{\pi}{2} -r- \varepsilon\big)  n\sqrt{\frac{(1-\lambda_n^2)}{8}} > C \]
for some sufficiently small $\varepsilon >0$. By construction, $D^{+}_{\bott_1}$ and $D^{+}_{\bott_2}$ coincide on the $(\frac{\pi}{2} -r- \varepsilon)$-neighborhood of $\double S\backslash X_\varepsilon$.  It follows from Theorem \ref{thm:relative} that 
\[  \ind(D^{+}_{\bott_1}) - \ind(D^{+}_{\bott_2}) = 0. \]
On the other hand, by the Atiyah-Singer index theorem \cite{MAIS68b}, we have 
\begin{align*}
	\ind(D^{+}_{\bott_1}) - \ind(D^{+}_{\bott_2}) 
	&  = (1 - \deg(\Phi)) \cdot  \ind(D_{E_0^+}^{\mathbb S^n}) \in K_0(\cpto) = \mathbb Z. 
\end{align*}
Since $\ind(D_{E_0^+}^{\mathbb S^n})\neq 0$ (cf. \cite[equation (4.7)]{MR1600027}),  we conclude that 
\[ 1 -\deg{\Phi}  = 0.\] 
This contradicts the fact that $\deg{\Phi} \neq 1$. This finishes the proof for the even dimensional case.

	Now let us prove the theorem in the odd dimensional case. Since the key ideas are similar to the even dimensional case, we shall be brief.  Again consider the canonical embedding of the unit sphere $\sphere^n$ inside the Euclidean space $\mathbb R^{n+1}$. Let $V = \mathbb R^{n+1}\times \ccl_{n+1}$ be the canonical $\ccl_{n+1}$-Dirac bundle over $\mathbb R^{n+1}$.  Denote by $v$ the outward unit normal vector field  of $\sphere^n$ in $\mathbb R^{n+1}$. Since $n+1$ is even in the current case, there is a canonical $\mathbb Z/2$-grading of $V$ induced by the following element: 
\[  \omega = i^{\frac{n+1}{2}} c_1 \cdots c_{n+1} \]
where $c_j$ is the Clifford multiplication by the standard basis element $e_j \in \mathbb R^{n+1}$. Let us write $ V = V^+\oplus V^-$ for this $\mathbb Z/2$-decomposition. Note that  $V^+$ and $V^-$ are trivial vector bundles of the same dimension. In particular, we can choose a fixed unitary matrix $U\in U_n(\mathbb C)$ to fiberwise identify $V^+$ and $V^-$.  Then a Bott element $\bottone$---a nonzero multiple of a generator of $K^1(\sphere^n) = K_1(C(\sphere^{n}))$---is given by the unitary 
\begin{equation}\label{eq:oddbott}
  \bottone =  i\hspace{0.07em}\overbar{c}(v)
\end{equation}
where $\overbar{c}(v)$ is the Clifford multiplication of $v$ on $V^+ = \sphere^n\times \ccl_{n+1}^+$.

Similar to the proof of Theorem \ref{thm:band2}, let us consider the following Dirac-type operator on $\sphere^1\times \sphere^n$:
\begin{equation*}
	\slashed{D} = c \cdot \frac{d}{d t} +  D_t
\end{equation*}
where $c$ is the Clifford multiplication of the unit vector $d/dt$ and 
\begin{equation}\label{eq:specflow}
	D_t \coloneqq t  D^{\sphere^n} + (1-t) \bottone  D^{\sphere^n} \bottone^{-1}
\end{equation}
for each $t\in [0, 1]$. Here we have chosen the parametrization $\sphere^1= [0, 1]/\{0, 1\}$. By the Atiyah-Singer index theorem \cite{MAIS68b}, we have  
\[ \ind(\slashed{D}) = \int_{\sphere^n} \widehat A(\sphere^n)\wedge  \ch(\bottone) \neq 0 \] 
where $\widehat A(\sphere^n)$ is the $\widehat A$-form of $\sphere^n$ and $\ch(\bottone)$ is the odd-dimensional Chern character of $\bottone$. 

Let $X_\varepsilon $ and $\double S$ be as above. 
Also, denote by $W = \double S\times \ccl_{n+1}^+$ the pullback bundle of the trivial bundle $V^+$. Pull back the unitary $\bottone$ by the maps $\idsp\colon \double S\to \sphere^n$ and ${\Phi\circ \idsp\colon \double S\to \sphere^n}$ and denote the resulting  unitaries by $\bottu_1 = \idsp^\ast(\bottone)$ and  $\bottu_2 = (\Phi\circ \idsp)^\ast(\bottone)$ respectively.  By construction, we have   $\bottu_1 = \bottu_2$ on the \mbox{$(\frac{\pi}{2} - r -\varepsilon)$-neighborhood} of $\double S\backslash X_\varepsilon$.  

Consider the Dirac-type operators 
\[ 	\slashed{D}^X_{\bottu_j} = c \cdot \frac{d}{d t} +  D_{\bottu_j, t}  \]
on $\sphere^1 \times \double S$, where 
\[D_{\bottu_j, t} = t  D + (1-t) \bottu_j  D \bottu_j^{-1}, \] for $j =1, 2$. In particular, $\slashed{D}^X_{\bottu_1}$ and $\slashed{D}^X_{\bottu_2}$ coincide on  the $(\frac{\pi}{2} - r -\varepsilon)$-neighborhood  of ${\sphere^1\times (\double S\backslash X_\varepsilon)}$.  Now the same calculation from the proof of Theorem \ref{thm:band2} shows that

\[ \slashed{D}_{\bottu_j}^2  = -\frac{d^2}{dt^2} + D_{\bottu_j, t}^2  + c [ D, \bottu_j]\bottu_j^{-1} \]
with 
\[ D_{\bottu_j, t}^2 = t  D^2 + (1-t)\bottu_j D^2 \bottu_j^{-1} + t(1-t) [ D, \bottu_j ][ D,\bottu_j^{-1}], \]
cf. line \eqref{eq:diracsquare}. It follows that 
\begin{align*}
D_{\bottu_j, t}^2 
	& \geq \frac{n^2}{4}  - \frac{n^2}{4}\lambda_n^2  = (1-\lambda_n^2)\frac{n^2}{4} \textup{ on } C_c^\infty(\mathbb S^1 \times X_\varepsilon, \spinb\otimes W).
\end{align*}

Similar to the rescaling argument from \eqref{eq:rescale}, for each $\mu >0$, we define the rescaled version of 	$\slashed{D}_{\bottu_j}$ to be  
\[ 	\slashed{D}_{\bottu_j, \mu} = c \cdot \frac{d}{d t} +  \mu D_{\bottu_j, t}.  \] The same calculation from above shows that 
\[ 	(\slashed{D}_{\bottu_j, \mu})^2 \geq \mu^2 (1-\lambda_n^2)\frac{n^2}{4} - \mu \frac{n}{2} \textup{ on } C_c^\infty(\mathbb S^1 \times X_\varepsilon, \spinb\otimes W).   \]
 
Now by applying the quantitative relative index theorem (Theorem \ref{thm:relative}), the rest of the proof for the odd dimensional case proceeds in   the same way as the even dimensional case. We conclude  that 
\[ \lambda_n \geq \sqrt{1 -  \frac{4C^2}{(\frac{\pi}{2} - r)^2 n^2}}  \]
where $C$ is a universal constant from Theorem \ref{thm:relative-intro}. This completes the proof of the theorem. 
	\end{proof}

In the case where $n=\dim \sphere^n$ is odd , the same proof for the odd dimensional case of  Theorem \ref{thm:sphere-ad} in fact also proves the $\lambda$-Lipschitz rigidity of positive scalar curvature metrics on $\sphere^n\backslash \Sigma$ when $\Sigma$  is a  subset of $\sphere^n$ contained in a pair of antipodal geodesic balls of radius $r < \frac{\pi}{6}$,

\begin{theorem}\label{thm:sphere-odd2}
	Let $\sphere^n = \sphere^{2k+1}$ be an odd dimensional standard unit sphere. Let $\Sigma$  be a subset of the standard unit sphere $\sphere^n$ contained in a pair of antipodal geodesic balls of radius $r < \frac{\pi}{6}$.  Let $(X, g_0)$ be the standard unit sphere $\sphere^n$ minus $\Sigma$. If a \textup{(}possibly incomplete\textup{)} Riemannian metric $g$ on $X$ satisfies that  
	\begin{enumerate}[$(1)$]
		\item there is a $\lambda_n$-Lipschitz  homeomorphism $\varphi\colon (X, g) \to (X, g_0)$,  	 
		\item and $\Sc(g) \geq n(n-1) = \Sc(g_0)$, 
	\end{enumerate}
	then 
	\[ \lambda_n \geq \sqrt{1 - \frac{C_r}{n^2}} \]
	where 
	\[ C_r = \frac{4C^2}{(\frac{\pi}{6} - r)^2}  \]
	and $C$ is a universal constant from Theorem $\ref{thm:relative}$.
\end{theorem}

\begin{proof}
	By assumption, $\Sigma$  is  contained in a pair of antipodal geodesic balls of radius $r < \frac{\pi}{6}$.  Let  $\underline{X}_{\varepsilon}\subset \sphere^n \backslash \Sigma$ be the closed subset of $\sphere^n$ with two antipodal open geodesic balls of radius $(r+\frac{\varepsilon}{2})$ removed.  By the proof of Lemma \ref{lm:wrap-antipodal},  there exists a  smooth distance-contracting map $\Phi\colon \sphere^n \to \sphere^n$ such that 
	\begin{enumerate}
		\item[(1)] $\Phi$ equals either the identity map or the antipodal map on each path component of the $(\frac{\pi}{6} -r- \varepsilon)$-neighborhood of $\sphere^n\backslash \underline{X}_{\varepsilon}$; 
		\item[(2)] and $\deg(\Phi) \neq 1$. 
	\end{enumerate}
 We view $X_\varepsilon$  as a (topological) subset of the $n$-dimensional sphere and  extend the Riemannian metric $g$ on $X_\varepsilon$ to a Riemannian metric on the sphere. Let us denote by $\double S$ the resulting $n$-dimensional sphere with this new metric $g_{\double S}$. Moreover, let $\bottone$ be the Bott unitary defined in line \eqref{eq:oddbott}.
	
Similar to the proof for the odd dimensional case of Theorem \ref{thm:sphere-ad2}, we pull back the unitary $\bottone$ by the maps $\idsp\colon \double S\to \sphere^n$ and ${\Phi\circ \idsp\colon \double S\to \sphere^n}$ and denote the resulting  unitaries by $\bottu_1 = \idsp^\ast(\bottone)$ and  $\bottu_2 = (\Phi\circ \idsp)^\ast(\bottone)$ respectively.  By construction, we have   $\bottu_1 = \pm \bottu_2$ on the \mbox{$(\frac{\pi}{6} - r -\varepsilon)$-neighborhood} of $\double S\backslash X_\varepsilon$.  

Consider the Dirac-type operators 
\[ 	\slashed{D}^X_{\bottu_j} = c \cdot \frac{d}{d t} +  D_{\bottu_j, t}  \]
on $\sphere^1 \times \double S$, where 
\[D_{\bottu_j, t} = t  D + (1-t) \bottu_j  D \bottu_j^{-1}, \] for $j =1, 2$. In particular, $\slashed{D}^X_{\bottu_1}$ and $\slashed{D}^X_{\bottu_2}$ coincide on  the $(\frac{\pi}{6} - r -\varepsilon)$-neighborhood  of ${\sphere^1\times (\double S\backslash X_\varepsilon)}$,  since $\bottu_1 = \pm \bottu_2$ on the \mbox{$(\frac{\pi}{6} - r -\varepsilon)$-neighborhood} of $\double S\backslash X_\varepsilon$.  

 Now the rest of the proof proceeds the same way as the proof for the odd dimensional case of Theorem \ref{thm:sphere-ad2}. We omit the details. This finishes the proof. 
\end{proof}

 As a consequence of Theorem \ref{thm:sphere-ad2}, we have the following $\lambda$-Lipschitz rigidity theorem for hemispheres.    This answers (asymptotically)  an open question of Gromov on the sharpness of the constant $\lambda_n$ for the $\lambda_n$-Lipschitz rigidity of positive scalar curvature metrics on hemispheres \cite[section 3.8]{Gromov:2019aa}. 

\begin{theorem}[Theorem \ref{thm:hemisphere}]\label{thm:hemisphere2}
	Let $(X, g_0)$ be the standard unit hemisphere $\sphere^n_+$. If a  Riemannian metric $g$ on $X$ satisfies that  
	\begin{enumerate}[$(1)$]
		\item there is a $\lambda_n$-Lipschitz  homeomorphism $\varphi\colon (X, g) \to (X, g_0)$,  	 
		\item and $\Sc(g) \geq n(n-1) = \Sc(g_0)$, 
	\end{enumerate}
	then 
	\[ \lambda_n \geq  (1-\sin \frac{\pi}{\sqrt{n}}) \sqrt{1 - \frac{8C^2}{\pi^2 n}} \]
	where  $C$ is a universal constant from Theorem $\ref{thm:relative-intro}$. Consequently, the lower bound for $\lambda_n$ approaches $1$, as $n\to \infty$. 	
\end{theorem}

This theorem is asymptotically optimal in the sense that   the lower bound for $\lambda_n$ becomes sharp,  as $n = \dim\sphere^n\to \infty$. In particular, it  significantly improves the lower bound for  $\lambda_n$ in Corollary \ref{cor:hemisphere} when $n = \dim\sphere^n$ is  large. 

\begin{proof}[Proof of Theorem \ref{thm:hemisphere2}]
Let $Y$ be the subspace of the standard unit sphere $\sphere^n$ with an open geodesic ball of radius $(\frac{\pi}{2}- \frac{\pi}{\sqrt{n}})$ removed. It is not difficult to see there exists a ${(1-\sin \frac{\pi}{\sqrt{n}})^{-1}}$-Lipschitz 	homeomorphism from $\sphere^n_+$ to $Y$. By composing with the map $\varphi\colon (X, g) \to (X, g_0)$, we obtain a $\lambda_n (1-\sin \frac{\pi}{\sqrt{n}})^{-1}$-Lipschitz  homeomorphism $\varphi\colon (X, g) \to (Y, g_0)$. It follows from Theorem \ref{thm:sphere-ad2} that 
\[ \lambda_n (1-\sin \frac{\pi}{\sqrt{n}})^{-1} \geq \sqrt{1 - \frac{C_r}{n^2}} \]
where\footnote{If $n=\dim \sphere^n$ is odd,  we can set   
	\[ C_r = \frac{4C^2}{(\frac{\pi}{2} - r)^2} \textup{ instead of } \frac{8C^2}{(\frac{\pi}{2} - r)^2}. \]  } 
\[ C_r = \frac{8C^2}{(\frac{\pi}{2} - r)^2} \textup{ and }  r = \frac{\pi}{2}- \frac{\pi}{\sqrt{n}}. \]
In other words, we have 
\[ \lambda_n \geq  (1-\sin \frac{\pi}{\sqrt{n}}) \sqrt{1 - \frac{8C^2}{\pi^2 n}}. \]
This finishes the proof. 
\end{proof}

\appendix

		\section{Friedrichs extensions in the maximal group  $C^\ast$-algebra setting }\label{sec:friedrichs}
		In this section, we show the existence of Friedrichs extensions of \mbox{semibounded} symmetric operators in the maximal group $C^\ast$-algebra setting. 
		
		More precisely, suppose $X$ is an $n$-dimensional compact spin manifold with boundary or corners. 
		Let $\spinb$ be the $\cl_n$-Clifford bundle over $X$ and $D$ a $\cl_n$-linear Dirac-type operator acting on $\spinb$. Let $\widetilde X$ be a Galois $\Gamma$-covering space of $X$ and  $\widetilde \spinb$ (resp. $\widetilde D$) the lift of $\spinb$ (resp. $D$).  By the Lichnerowicz formula, we have 
		\[ \widetilde D^2 = \nabla^\ast \nabla + \mathcal R  \]
		where $\mathcal R$ is a symmetric bundle endomorphism of $\widetilde \spinb$. For the rest of this  section, let us assume there exists $\sigma >0$ such that 
		\[  \mathcal R(x) \geq \frac{(n-1)\sigma^2 }{n} \]
	for all $x\in \widetilde X$.

	\begin{definition}\label{def:sobnorm}
	We define $H^0_1(\interior{\widetilde X}, \widetilde \spinb)$ to be the completion  of $C_c^\infty(\interior{\widetilde X}, \widetilde \spinb)$ with respect to the Sobolev norm 
	\begin{equation}\label{eq:sobnorm2}
		\|v\|_{1} =  \Big( \int_{\interior{\widetilde X}}|v|^2 +  \int_{\interior{\widetilde X}} |\nabla v|^2\Big)^{1/2}.
	\end{equation} 
\end{definition}
		
		Consider the Friedrichs extension $F$ of $\widetilde D^2$ on $L^2(\interior{\widetilde X}, \widetilde \spinb)$ with respect to the domain $H_0^1(\interior{\widetilde X}, \widetilde \spinb)$. Let us write  
		\begin{equation}\label{eq:fried}
\fried \coloneqq  F^{1/2}.
		\end{equation} Then it is known that the domain $ \dom(\fried)$  of $\fried$ is precisely $H_0^1(\interior{\widetilde X}, \widetilde \spinb)$, cf. \cite[chapter 8, proposition 1.10]{MR2743652}.
		Furthermore, the wave operator $e^{it\fried}$ associated to $\fried$ has finite propagation, cf.  \cite[chapter 2, section 6]{MR2744150}.

Let $\mathcal F$ be a fundamental domain of $\widetilde X$ under the $\Gamma$-action and $\rho$ the characteristic function of $\mathcal F$. Define  $\rho_\gamma$ to be the $\gamma$-translation of $\rho$, that is, 
\[\rho_\gamma(x) = \rho(\gamma^{-1}x). \] 
For a given $a \in \mathbb R$, let us write  $T = (\fried + ia )^{-1}$. We define 
\begin{equation}\label{eq:fund}
	T_\gamma =   \rho_\gamma \circ T \circ  \rho.
\end{equation}

In the following,  we shall fix a length metric $l \colon \Gamma \to \mathbb R_{\geq 0}$ on $\Gamma$. Then there exist $A_\Gamma >0$ and $B_\Gamma >0$  such that 
\begin{equation}\label{eq:quasi-iso}
	A_\Gamma^{-1}\cdot  \dist(\gamma \mathcal F, \mathcal F) - B_\Gamma \leq l(\gamma) \leq A_\Gamma \cdot \dist(\gamma \mathcal F, \mathcal F) + B_\Gamma
\end{equation}
for all $\gamma\in \Gamma$, where  $\dist(\gamma \mathcal F, \mathcal F)$ is the distance between the two sets $\gamma \mathcal F$ and $\mathcal F$ measured with respect to the given Riemannian metric on $\widetilde X$. 
\begin{lemma}\label{lm:opnorm}
	Let $T = (\fried + i a)^{-1}$ as above. Then there exists a constant $C_1>0$ such that 
	\[  \|T_\gamma\| \leq C_1 e^{ - |a| \cdot A_\Gamma^{-1}\cdot l(\gamma)},  \]
	for all $\gamma\in \Gamma$, where $\|T_\gamma\|$ is the operator norm of the operator 
	\[ T_\gamma \colon L^2(\interior{\widetilde X},  \widetilde{\spinb}) \to L^2(\interior{\widetilde X},  \widetilde{\spinb}). \] 
\end{lemma}
\begin{proof}
	If $a=0$, the lemma is trivial. Without loss of generality, let us assume $a>0$, since the case where $a<0$ can be treated exactly the same way. 	The Fourier transform of $f(x) = (x + ia)^{-1}$ is 
	\[\hat{f}(\xi) = \frac{1}{\sqrt{2\pi}}\int_{\mathbb R} f(x) e^{-i\xi x}\dif x = -i \sqrt{2\pi} e^{-a\xi}\theta(\xi)   \]
	where $\theta$ is the unit step function 
	\[ \theta(\xi) = \begin{cases}
		0 & \textup{ if $\xi< 0$,} \\
		1 & \textup{ if $\xi \geq 0$.}
	\end{cases} \] 
	In particular, $\hat f$ and all of its derivatives are smooth away from $\xi = 0$ and decay exponentially as $|\xi|\to \infty$. 
	
	Let $\varphi$ be a  smooth function on $\mathbb R$ with $0\leq \varphi(x)  \leq 1$ such that 
	$\varphi(x)  = 1$ for all $|x|\geq 2$ and $\varphi(x) = 0$ for all $|x|\leq 1$.  For each $t > 0 $, we define $h_t$ to be the function on $\mathbb R$ whose Fourier transform is 
	\[ \hat h_t(\xi)  = \varphi(t^{-1}\xi) \hat f(\xi).  \]
	For each fixed $t>0$, we apply functional calculus to define  the operator $ R \coloneqq  h_t(\fried)$. 
	We have 
	\[  R(v) = \frac{1}{\sqrt{2\pi}} \int_{\mathbb R} \varphi(t^{-1}\xi) \hat f(\omega) e^{i\xi \fried} v  \dif\xi \]
	for all $v\in L^2(\interior{\widetilde X}, \widetilde \spinb)$. Define
	\[ R_\gamma = \rho_\gamma \circ R\circ  \rho. \] We see that there exists a constant $C_2>0$ such that    \begin{align*}
		\|R_\gamma\|\leq |\rho_\gamma\| \cdot \|R\| \cdot  \| \rho\|\leq \frac{1}{\sqrt{2\pi}}\int_{\mathbb R} \varphi(t^{-1}\xi) |\hat f(\xi)| \dif \xi \leq C_2 e^{-at}
	\end{align*}
	for all $\gamma\in \Gamma$.
	By finite propagation of the wave operator $e^{is\fried}$,  it follows that 
	\[ T_\gamma =  R_\gamma \]
	for all but finitely many $\gamma\in \Gamma$. More precisely, we have $T_\gamma =  R_\gamma$ for all $\gamma$ with $l(\gamma) \geq A_\Gamma\cdot t + B_\Gamma$. By varying $t$, it is not difficult  to see that there exists a constant $C>0$ such that 
	\[  \|T_\gamma\| \leq C e^{ - a \cdot A_\Gamma^{-1}\cdot l(\gamma)} \] 
	for all $\gamma \in \Gamma$. 
\end{proof}

		For any given $a\neq 0\in \mathbb R$, the range $\Ran(\fried+ ia)$ of $\fried + ia$ is equal to the whole space $ L^2(\interior{\widetilde X},  \widetilde{\spinb}) $, since $\fried$ is self-adjoint. In particular, for each $f\in C_c^\infty(\interior{\widetilde X}, \widetilde \spinb)$, there is   $v =   (\fried + ia)^{-1}(f)\in  \dom(\fried)$ such that 
\[ (\fried+ ia) (v) = f.  \]

 We define $\mathcal L^2_{C^\ast_{\max}(\Gamma; \mathbb R)}$ to be the completion of $C_c^\infty(\interior{\widetilde X}, \widetilde \spinb)$ with respect to the  
following Hilbert $C^\ast_{\max}(\Gamma; \mathbb R)$-inner product 
\begin{equation}\label{eq:maxinner}
\llangle f_1, f_2\rrangle \coloneqq  \sum_{\gamma\in \Gamma} \langle f_1,  \gamma f_2  \rangle \gamma  \in  C^\ast_{\max}(\Gamma; \mathbb R) 
\end{equation} 
for all $f_1, f_2\in C_c^\infty(\interior{\widetilde X}, \widetilde \spinb)$, where 
\[ \langle f_1,  \gamma f_2  \rangle = \int_{\interior{\widetilde X}} \langle {f}_1(x), f_2(\gamma^{-1}x)\rangle dx. \]
Let us define
\[ \|v\|_{\max} \coloneqq  \llangle v, v\rrangle^{1/2} \]
for $v\in \mathcal L^2_{C^\ast_{\max}(\Gamma; \mathbb R)}$.

		The following lemma is a consequence of Lemma \ref{lm:opnorm}.
		\begin{lemma}\label{lm:decay}
			If $|a|$ is sufficiently large, then for every $f\in C_c^\infty(\interior{\widetilde X}, \widetilde \spinb)$,  the  element	$ v = (\fried + ia)^{-1}(f)$ 
			lies in $\mathcal L^2_{ C^\ast_{\max}(\Gamma; \mathbb R)}$  
		\end{lemma}
		\begin{proof}
			Let $\{ 
			\rho_\gamma\}_{\gamma\in \Gamma}$ be the characteristic functions as above. We have 
			\[ v = \sum_{\gamma \in \Gamma } \rho_\gamma v. \]
			Clearly, each $\rho_\gamma v$ lies in $\mathcal L^2_{C^\ast_{\max}(\Gamma; \mathbb R)}$, since  $\rho_\gamma v$ is supported on a metric ball of bounded radius.

			By Lemma \ref{lm:opnorm}, a straightforward calculation shows that there exists a constant\footnote{The constant $C_f$ depends on $f$. More precisely, the constant $C_f$ depends on the diameter of the support of $f$.} $C_f >0$ such that 
			\[ \langle v, \beta v\rangle \leq C_f \cdot e^{-|a|\cdot A_\Gamma^{-1}\cdot  l(\beta)}\cdot  \|f\|_{L^2}, \]
			where $l(\beta)$ is the word length of $\beta$  and the constant $A_\Gamma^{-1}$ is defined in line $\eqref{eq:quasi-iso}$. 
			Since the group $\Gamma$ has at most exponential growth, that is, there exist numbers $K_\Gamma >0$ and $C_2$ such that 
			\[ \#\{\alpha\in \Gamma \mid l(\alpha) \leq n\} \leq  C_2 e^{K_\Gamma \cdot n} \]
			for all $n\in \mathbb N$.  It follows that  
			\[ \|v\|_{ \max}^2 = \llangle v, v \rrangle = \sum_{\beta\in \Gamma} \langle v,  \beta v  \rangle\, \beta  \in C_{\max}^\ast(\Gamma; \mathbb R) \]
			as long as $|a|$ is sufficiently large. This finishes the proof. 
		\end{proof}

		For each $a\in \mathbb R$ such that $|a|$ is sufficiently large, consider the operator  
		\[ \fried_{\max} \colon \mathcal L^2_{ C^\ast_{\max}(\Gamma; \mathbb R)}\to \mathcal L^2_{ C^\ast_{\max}(\Gamma; \mathbb R)}  \]
		defined by setting $ \fried_{\max}(v) = \fried(v)$ on the domain 
		\[ \dom(\fried_{\max}) =  C_c^\infty(\interior{\widetilde X}, \widetilde \spinb) + (\fried + ia)^{-1}(C_c^\infty(\interior{\widetilde X}, \widetilde \spinb) ) \]
		where  $(\fried + ia)^{-1}(C_c^\infty(\interior{\widetilde X}, \widetilde \spinb))$ consists of 
		\[  \{ v\in \mathcal L^2_{C^\ast_{\max}(\Gamma; \mathbb R)} \mid  v = (\fried + ia)^{-1} f \textup{ for some } f\in  C_c^\infty(\interior{\widetilde X}, \widetilde \spinb) \}.\] 
		As an immediate consequence of Lemma \ref{lm:decay}, we see that 
		$\fried_{\max}$ is well-defined. Moreover,  $\fried_{\max}$ is an unbounded symmetric operator, since $\fried$ is symmetric with respect to the inner product from line $\eqref{eq:maxinner}$.
		\begin{lemma}\label{lm:regular}
			For each $a\in \mathbb R$ such that $|a|$ is sufficiently large, the closure of $\fried_{\max}$ is regular and self-adjoint. 
		\end{lemma}
		\begin{proof}
			By construction, the operator $(\fried_{\max} + ia)$ has a dense range. By \cite[lemma 9.7 \& 9.8]{EL95}, we conclude that the closure  of $\fried_{\max}$ is regular and self-adjoint. 
		\end{proof}

		We have the following main result of this section.  
		
		\begin{proposition}

			Suppose $X$ is an $n$-dimensional compact spin manifold with boundary or corners. 
			Let $\spinb$ be the $\cl_n$-Clifford bundle over $X$ and $D$ a $\cl_n$-linear Dirac-type operator acting on $\spinb$. Let $\widetilde X$ be a Galois $\Gamma$-covering space of $X$ and  $\widetilde \spinb$ (resp. $\widetilde D$) the lift of $\spinb$ (resp. $D$). Let $\mathcal R$ be the symmetric bundle endomorphism of $\widetilde \spinb$ appearing in the following Lichnerowicz  formula 
			\[ \widetilde D^2 = \nabla^\ast \nabla + \mathcal R.  \]
			Assume there exists  $\sigma >0$ such that
			\[  \mathcal R(x) \geq \frac{(n-1)\sigma^2}{n} \]
			for all $x\in \widetilde X$.	
			Then there exists a self-adjoint Friedrichs extension $F_{\max}$ of $\widetilde D^2$:
			\[ F_{\max} \colon \mathcal L^2_{ C^\ast_{\max}(\Gamma; \mathbb R)}  \to \mathcal L^2_{ C^\ast_{\max}(\Gamma; \mathbb R)} \] 
			such that the following are satisfied:
			\begin{enumerate}[$(1)$]
				\item  $\|F_{\max} (f)\|_{\max} \geq \sigma^2 \|f\|_{ \max}$ for all $f\in  \dom(F_{\max})$,
				\item and the wave operator $e^{is\dirac}$ has finite propagation, where $\dirac =  (F_{\max})^{1/2}$. 
			\end{enumerate}
		\end{proposition}  
		\begin{proof}
			Fix $a\in \mathbb R$ such that $|a|$ is sufficiently large. Let $\fried_{\max}$ be the operator from Lemma $\ref{lm:regular}$.  We denote its closure by $\dirac$. Let us define  $F_{\max} =\dirac^2$. By construction, the wave operator $e^{is\dirac}$ has finite propagation, since $e^{is|\widetilde D|}$ has finite propagation, where $|\widetilde D|$ is the operator defined in line \eqref{eq:fried}.
			
		To prove the proposition, it suffices to verify 
				\[  \|\dirac (f)\|_{ \max} \geq \sigma \|f\|_{\max} \]
	  for all $v\in \dom(|\widetilde D|_{\max} ) = C_c^\infty(\interior{\widetilde X}, \widetilde \spinb) + (\fried + ia)^{-1}(C_c^\infty(\interior{\widetilde X}, \widetilde \spinb) )$.  Given 
			\[ v =  f_1 + (\fried + ia)^{-1} f_2\] 
		for some $f_1, f_2\in C_c^\infty(\interior{\widetilde X}, \widetilde \spinb) $, we have 
			\[  (\fried_{\max}+ ia)v  = (|\widetilde D| + ia)f_1 + f_2, \]
			which implies 
			\begin{align*}
				\fried_{\max}(v) & = - ia v + (\fried+ ia)f_1 + f_2 \\
				& = -ia f_1 -ia  (\fried + ia)^{-1} f_2 + (\fried + ia)f_1 + f_2. 
			\end{align*}  
			It follows that for  each $v\in \dom(|\widetilde D|_{\max} )$, we have $\fried_{\max}(v)\in \dom(|\widetilde D|_{\max} ).$  We conclude that 
			\begin{align*}
				\|\dirac (v)\|^2_{\max} =  & \llangle |\widetilde D|_{\max}(v), |\widetilde D|_{\max}(v) \rrangle \\
				= &     \llangle  \fried v, \fried v \rrangle \\ 
				= &   \llangle F  v,  v \rrangle 
			\end{align*}  
		all $v\in \dom(|\widetilde D|_{\max} )$. 
		
		 Note that $\dom(|\widetilde D|_{\max} ) = C_c^\infty(\interior{\widetilde X}, \widetilde \spinb) + (\fried + ia)^{-1}(C_c^\infty(\interior{\widetilde X}, \widetilde \spinb) )$ is contained in $ H_0^1(\interior{\widetilde X}, \widetilde \spinb)$, since  $\dom(\fried)  = H_0^1(\interior{\widetilde X}, \widetilde \spinb)$. It follows that 
		 \[   \llangle F  v,  v \rrangle =  \llangle \widetilde D^2  v,  v \rrangle = \llangle \widetilde D  v,  \widetilde D v \rrangle \]
		for all $v\in \dom(|\widetilde D|_{\max} )$. 
			By the Cauchy–Schwarz inequality, we have
		\[ \llangle \widetilde D  v,  \widetilde D v \rrangle\leq n \llangle \nabla v, \nabla v\rrangle   \]
		for all $v\in \dom(|\widetilde D|_{\max} )$,  where $n = \dim X$. Since 
		\[   \llangle \widetilde D  v,  \widetilde D v \rrangle = \llangle \widetilde D^2  v,  v \rrangle= \llangle  \nabla^\ast \nabla v, v \rrangle + \llangle \mathcal R  v,  v \rrangle \]
		for  all $v\in \dom(|\widetilde D|_{\max} )$, it  follows that 
		\begin{equation}
			\frac{n-1}{n}\llangle \widetilde D  v,  \widetilde D v \rrangle \geq  \llangle \mathcal R v,  v \rrangle \geq \frac{(n-1)\sigma^2}{n} \llangle v, v\rrangle, 
		\end{equation} 
		for all $v\in \dom(|\widetilde D|_{\max} )$. To summarize, we have showed that 
		\[ 	\|\dirac (v)\|^2_{\max}  =  \llangle F  v,  v \rrangle \geq \sigma^2 \llangle v, v\rrangle  \]
		for all $v\in \dom(|\widetilde D|_{\max} )$. 
			This finishes the proof.

		\end{proof}

\section{An estimate of the universal constant $C$ in Theorem \ref{thm:relative-intro}}\label{app:estimate}

\begin{center}
(by Jinmin Wang and Zhizhang Xie)
\end{center}
\vspace{0.2cm} 

In this appendix, we give a numerical estimate of the universal constant $C$ that appeared in Theorem \ref{thm:relative-intro}. Part of the numerical computation is done with the assistance of MATLAB.  We would like to thank Li Zhou for many helpful comments on the MATLAB programming.   Throughout this section,  we use the following convention of the Fourier transform
$$\hat f(\xi)=\int_\R f(x)e^{-ix\xi}dx$$
and the inverse Fourier transform 
$$f(x)=\frac{1}{2\pi}\int_\R \hat f(\xi)e^{ix\xi}d\xi.$$

\subsection{Increasing normalizing functions} 

Let $\chi\colon \mathbb R\to [-1, 1]$ be a normalizing function, that is, $\chi$ a continuous odd function such that $\lim_{x\to \pm \infty} \chi(x) = \pm 1.$  In addition, let us assume $\chi$ is an increasing function in this subsection. This additional assumption is not necessary, but it makes the numerical estimates somewhat easier. The method in this subsection is inspired by the work of Slepian \cite{MR710468}.  We will discuss general normalizing functions in the next subsection and carry out some estimates via a different method. 

 Note that given $\varepsilon >0$,  there exists $\sigma>0$ such that 	$$\sup_{|x|\geq\sigma}|1-\chi(x)^2|<\varepsilon.$$
 Here $\sigma$ has the same meaning as the spectral gap $\sigma$ of $D$ on a subspace $Y\backslash K$ in Lemma \ref{lm:fp} and Lemma \ref{lm:PtD}.  The key step for estimating the universal constant $C$ from Theorem \ref{thm:relative-intro} is to answer the following minimization  question.\footnote{Alternatively, we can fix the value of the spectral gap $\sigma$ and try to minimize the propagation, that is, the support of the Fourier transform $\hat \chi$. The two approaches are essentially equivalent.} 
\begin{question}\label{q1}
	Given $\varepsilon >0$, find the minimum  value of $\sigma$ among all normalizing functions $\chi$ whose  Fourier transform $\hat\chi$ is supported on $[-2,2]$. 
\end{question}

Let us set 
\[ f=\frac{\chi'}{2}, \]
where $\chi'$ is the derivative of $\chi$. Here the factor $1/2$ is only introduced as a normalizing constant.   Since $\chi$ is assumed to be increasing,  $f$ is a positive even function such that its Fourier transform $\hat f$ is supported on $[-2,2]$ and $\hat f(0)=1$. As another simplification,\footnote{Such an simplification does not really affect the final estimates at the end of this subsection.} let us only consider  $f$ such that $f=h^2$, where $h$ is a function with its Fourier transform $\hat h$ supported on $[-1,1]$. Since $f=h^2$, we have $\hat f=\hat h*\hat h$, i.e.,
$$\hat f(\xi)=\int_\R \hat h(\xi-\eta)\hat h(\eta)d\eta.$$
It follows that 
$$\hat f(0)=\int_\R \hat h(-\eta)\hat h(\eta)d\eta=\|h\|_{L^2}^2,$$
where the second equality use the fact that $h$ is an even function.

Observe that the minimization question (Question \ref{q1}) is essentially equivalent to the following maximization question.  
\begin{question}
	Given $\sigma>0$, find the maximal value of
		\[ \int_{-\sigma}^\sigma h^2(x)dx  \]
		among all functions $h\colon \mathbb R\to \mathbb R$ with $\|h\|_{L^2} = 1$ such that its Fourier transform is supported on $[-1,1]$.  
\end{question}

Let $c_\sigma(x)$ be the characteristic function of $[-\sigma, \sigma]$, that is,
\[ c_\sigma(x) = \begin{cases}
	1 & \textup{ if }  |x|\leq \sigma, \\
	0 & \textup{ if }  |x|> \sigma. 
\end{cases}\] Then the Fourier transform of $c_\sigma$ is given by 
$$\hat c_\sigma(\xi)=\frac{2\sin(\sigma\xi)}{\xi}.$$
Note that
$$\int_{-\sigma}^{\sigma} f(x)dx=\int_\R f(x) c_\sigma(x) dx=
\frac{1}{2\pi}\int_\R \hat f(\xi)\hat{c}_\sigma (\xi) d\xi
=\int_\R \frac{\sin(\sigma\xi)}{\pi\xi}\hat f(\xi) d\xi.$$
Recall that $f = h^2$.  It follows that 
$$\int_{-\sigma}^{\sigma} f(x)dx= \int_\R \int_\R \frac{\sin(\sigma\xi)}{\pi\xi}\hat h(\xi-\eta)\hat h(\eta)d\eta d\xi.$$
We conclude that 
$$\int_{-\sigma}^\sigma h^2(s)ds =
\int_{-1}^{1}\int_{-1}^{1} \frac{\sin(\sigma(x+y))}{\pi(x+y)}\hat h(x)\hat h(y)dxdy, $$
since $\hat h$ is supported on $[-1,1]$. 

This naturally leads us to the integral operator $T_\sigma$ on $L^2([-1,1])$ given by
$$T_\sigma(\varphi)(x)=\int_{-1}^1 \frac{\sin(\sigma(x+y))}{\pi(x+y)} \varphi(y)dy$$
for all $\varphi\in L^2([-1, 1])$. 
Observe that  $T_\sigma$ is a bounded self-adjoint compact operator. Hence the maximum value of
\[ \int_{-\sigma}^\sigma h^2  \] 
 is  equal to the largest (absolute) eigenvalue of $T_\sigma$, which is precisely the  operator norm $\|T_\sigma\|$ of $T_\sigma$. Moreover, the maximum is achieved when $h$ is a corresponding eigen-function of the largest (absolute) eigenvalue of $T_\sigma$. This eigen-function   can be numerically estimated by considering the following iteration
$$\varphi_{n} =  \|T_\sigma \varphi_{n-1}\|^{-1}_{L^2}  \cdot  T_\sigma(\varphi_{n-1}),$$
with $\varphi_0 \equiv 1$. In particular, $\varphi_{n}$ converges to an eigen-function of  the largest (absolute) eigenvalue of $T_\sigma$.

The detailed estimates can be carried out as follows. 

\begin{enumerate}
	\item Given any normalizing function $\chi\colon \mathbb R\to [-1, 1]$, let $p_t$ be the idempotent from line \eqref{eq:idem}.  By Lemma \ref{lm:quasi-idem}, we need to estimate the operator norm $\|p_t\|$ of $p_t$. It amounts to estimating the operator norm  of the following idempotents
	$$\mathbb P_a=\begin{pmatrix}
		a^2(2-a^2)& (2-a^2)(1-a^2)a\\
		a(1-a^2)& (1-a^2)^2
	\end{pmatrix},$$
for all $a\in [-1, 1]$. A numerical estimate shows that 
	$$\sup_{a\in [-1, 1]} \|\mathbb P_a\|^2 = \sup_{a\in[-1,1]}\{\text{eigenvalues of } \mathbb P^\ast_a\mathbb P_a \} \approx (1.04015)^2.$$
We conclude that 
	\begin{equation}\label{eq:idemnorm}
\|p_t\|\approx 1.04015.
	\end{equation}
	\item Again by Lemma \ref{lm:quasi-idem} and part (3) of Lemma \ref{lm:PtD}, we need to estimate the operator norm 
	\[ \|p_t - \begin{psmallmatrix} 1 & 0 \\ 0 & 0 \end{psmallmatrix} \|\]
	on some subspace of $Y$. This amounts to finding $\theta>0$ such that  
\begin{align}
\|\mathbb P_a - \begin{psmallmatrix} 1 & 0 \\ 0 & 0 \end{psmallmatrix}\|  & =\Big\|\begin{psmallmatrix}
	-(1-a^2)& (2-a^2)(1-a^2)a\\
	a(1-a^2)& (1-a^2)^2
\end{psmallmatrix}\Big\| \notag\\
& \leq \frac{1}{(2\|\mathbb P_a\|+2)} \frac{1}{4}\leq  \frac{1}{(2(1.04015)+2)} \frac{1}{4} = \frac{1}{16.3212}, \label{eq:small}
\end{align}
	for all $ a $ with $\theta \leq |a|\leq 1$. A numerical estimate shows that 
	\begin{equation}\label{eq:theta}
\theta \approx 0.96978.
	\end{equation}
	
	\item Now the last step is to find $\sigma$ such that 
	\[  \|T_\sigma\| = \theta \approx 0.96978.\]
	Again a numerical estimate shows that 
	$$\sigma\approx 2.86821.$$
\end{enumerate}

Now let $\chi\colon \mathbb R\to [-1, 1]$ be a normalizing function that fulfills the above estimates. Since the Fourier transform $\hat \chi$ is supported on $[-2, 2]$, it follows from the proof of Theorem \ref{thm:relative-intro} (cf. Theorem \ref{thm:relative}) that we can choose $r = 2\cdot 15 = 30$. It follows that  the universal constant $C$ in Theorem \ref{thm:relative-intro} satisfies 
\[ C\leq 2.86821 \cdot 2\cdot 15  = 86.0463. \]

\subsection{An improved estimate}

In this subsection, we shall improve our estimate for the universal constant $C$ by considering general normalizing functions. That is, the normalizing function $\chi$ is \emph{not} assumed to be increasing any more. 

Let $\sgn$ be the sign function, that is, 
\[ \sgn(x) = \begin{cases}
	1 & \textup{ if }  x\geq 0, \\
	-1 & \textup{ if }  x<0.
\end{cases}\]
 Its Fourier transform is given by 
\[ \widehat \sgn(\xi) =  \int_\R \sgn(x) e^{-ix\xi}dx = \frac{2}{i\xi}.\]
Our goal is to solve the following analogue of Question \ref{q1}.

\begin{question}\label{q3} For each  function $f$, let us define
	\[ \chi_f(x) =  \frac{1}{2\pi }\int_\R f(\xi)  \frac{2  }{i\xi} e^{ix\xi}d\xi. \]
	For given $\varepsilon_1>0$ and $\varepsilon_2>0$,  find the minimum value of $\sigma$ such that there exists an even function $f$ satisfying the following conditions: \begin{enumerate}[$(1)$]
		\item $f$ is supported on $[-2, 2]$, 
		\item $f(0) =1 $,
		\item 	$  \varepsilon_1 <  \chi_f(x) - 1 < \varepsilon_2$ for all $x\geq \sigma$, 
		\item and $-\varepsilon_2 <  \chi_f(x) + 1 < - \varepsilon_1$ for all $x\leq  -\sigma.$
	\end{enumerate}
\end{question}
Since $f$ is an even function, it follows that $\chi_f$ is an odd function. Furthermore, we have 
\[  \lim_{x\to \infty} \chi_f(x) = \int_{0}^\infty \chi'_f(t)dt = \frac{\widehat{\chi'_f}(0)}{2} = f(0) = 1. \] 
It follows that, under the above assumptions on $f$, the function  $\chi_f$ is a normalizing function.

As a simplification, we shall only consider even functions $f$ such that  the function $\chi_f$  takes values in $[-1.2, 1.2]$. Our numerical estimates show that such a simplification essentially does not result in any loss of generality. Under the extra assumption that $\chi_f$  takes values in $[-1.2, 1.2]$, the norm estimate in line \eqref{eq:idemnorm} will remain the same. Then based on the estimate in line \eqref{eq:small}, we can choose 
\[ \varepsilon_1 \approx 1 - 0.96978 =  0.03022,\] 
and 
\[ \varepsilon_2 \approx 1.02928 - 1 = 0.02928. \] 
Now to approximate the minimum value of $\sigma$ in Question \ref{q3}, we consider the function $f$ of the form
\[ f(\xi) = \sum_{k=0}^{n}  a_k\cos(\frac{k \pi \xi}{2})\] 
with real numbers $\{a_k\}_{0\leq k \leq n}$  to be determined. Note that 
\[ \varphi_k(x) \coloneqq  \frac{1}{2\pi } \int_\R \cos(\frac{k \pi \xi}{2})  \frac{2 }{i\xi} e^{ix\xi}d\xi  =  	\frac{1}{2\pi } \int_{2 \pi  k-2 x}^{2 \pi  k+2 x} \frac{ 2\sin t}{t}dt.  \]
So we can write 
\begin{equation}\label{eq:approx}
\chi_f(x) =  \sum_{k=0}^{n} a_k\varphi_k(x).
\end{equation} 
We see that Condition (2) in Question \ref{q3} becomes the following linear inequality: 	
\[   0.03022 <  \sum_{k=0}^{n} a_k\varphi_k(x) - 1 < 0.02928\]
for each $x\geq \sigma$.  
This reduces Question \ref{q3} into solving a system of linear inequalities. 

If we set $n =5$, then a numerical estimate shows that the following function 
\begin{align}
f(\xi)  \approx &     0.75382052 + 
0.25425247\cos(\frac{\pi\xi}{2}) + 0.0034679636 \cos(\pi\xi) \notag\\
&  -0.026352193
 \cos(\frac{3\pi\xi}{2}) + 0.024841712 \cos(2\pi\xi) \notag\\
 & -0.010030481  \cos(\frac{5\pi\xi}{2}) \label{eq:normal}
\end{align}
satisfies the conditions in Question \ref{q3}, and we can choose \[\sigma \approx 1.41356\] in this case. See Figure \ref{fig:norm} for the graph of the corresponding normalizing function $\chi_f$. 
\begin{figure}
	\includegraphics[height=5cm]{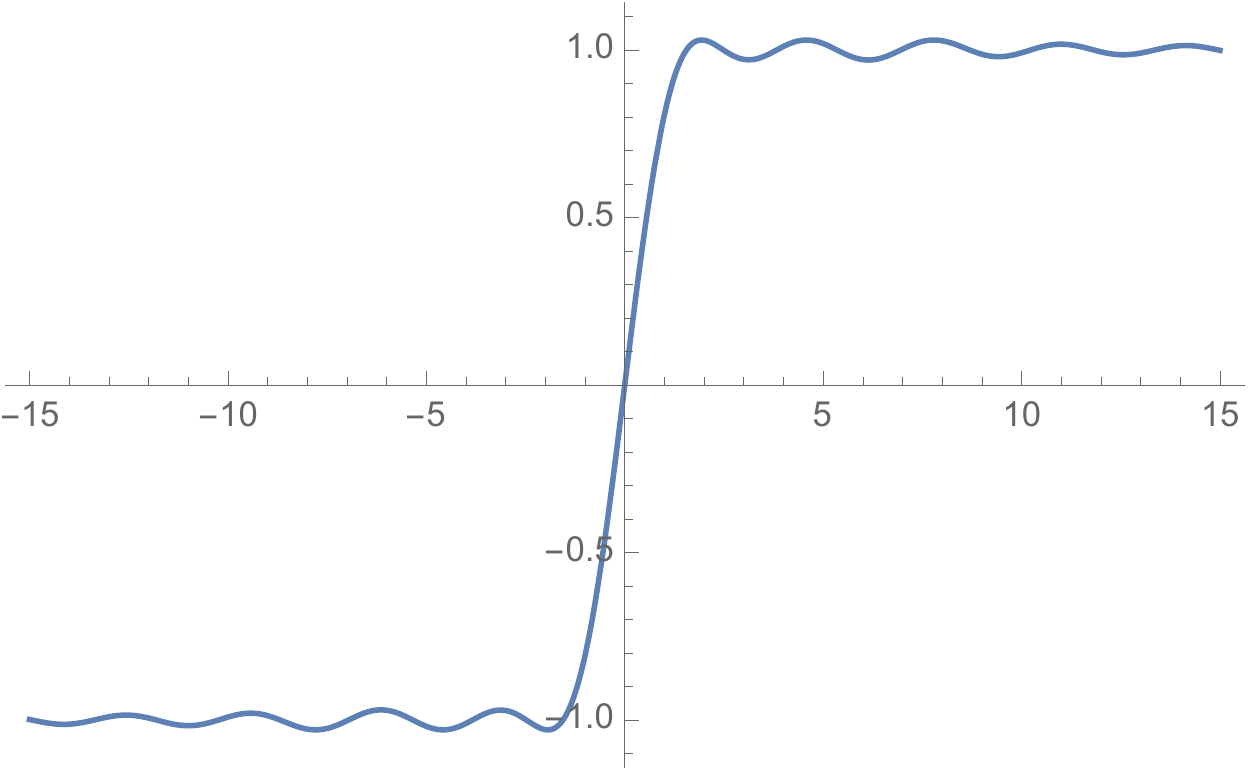}
	\caption{The graph of $\chi_f$ when $f$ is the function  in line \eqref{eq:normal}  }
	\label{fig:norm}
\end{figure}
  Further numerical estimates show that we can choose 
$$\sigma\approx 1.3633 \textup{ if we work with } n = 20,$$ 
and 
$$\sigma\approx 1.355 \textup{ if we work with } n = 50.$$  This suggests that $1.355$ is rather close to the actual minimum value of  $\sigma$ in Question \ref{q3}. In any case, we conclude the universal constant $C$ in Theorem \ref{thm:relative-intro} satisfies 
\[ C\leq 1.355 \cdot 2\cdot 15  = 40.65. \]

\end{document}